\documentclass[12pt, a4paper,reqno]{amsart}

\usepackage[T1]{fontenc}
\usepackage[utf8]{inputenc}

\usepackage{color}
\definecolor{bleu_sombre}{rgb}{0,0,0.6}  \definecolor{rouge_sombre}{rgb}{0.8,0,0}\definecolor{vert_sombre}{rgb}{0,0.6,0}

\usepackage[english]{babel}
\usepackage{geometry}

\usepackage{amsmath,amssymb,amsthm,graphicx,amsfonts,url,enumerate,dsfont,stmaryrd, mathrsfs, csquotes}
\usepackage{mathabx}

\usepackage{pgf,tikz,pgfplots}
\usetikzlibrary{arrows}
\definecolor{uuuuuu}{rgb}{0.26666666666666666,0.26666666666666666,0.26666666666666666}

\theoremstyle{plain}
\newtheorem{theorem}{{Theorem}}[section]
\newtheorem*{theorem*}{{Theorem}}
\newtheorem{proposition}[theorem]{Proposition}

\newtheorem*{proposition*}{Proposition}

\newtheorem*{corollary*}{Corollary}

\newtheorem{assumption}[theorem]{Assumption}
\newtheorem*{lemma*}{Lemma}
\theoremstyle{definition}

\newtheorem*{definition*}{Definition}
\theoremstyle{remark}
\newtheorem{remark}[theorem]{Remark}

\makeatletter

\@addtoreset{equation}{section}
\makeatother

\renewcommand{\leq}{\leqslant}	\renewcommand{\geq}{\geqslant}

\newcommand{\R}{\mathbb{R}}

\renewcommand{\Re}{\mathrm{Re}\,}

\title[Curved magnetic edge]{Tunneling effect induced by a curved magnetic edge}

\author{S{\o}ren Fournais}

\author{Bernard Helffer}

\author{Ayman Kachmar}

\address{Department of Mathematics, Aarhus University, Ny Munkegade 118, 8000 Aarhus~C, Denmark}
\email{fournais@math.au.dk}

\address{Universit\'e de Nantes, Laboratoire Jean Leray, Nantes, France }
\email{Bernard.Helffer@univ-nantes.fr}

\address{Lebanese University, Department of mathematics, 1700 Nabatieh, Lebanon}
\email{akachmar@ul.edu.lb}

\begin{document}

\begin{abstract}
Experimentally observed magnetic fields with nanoscale variations are theoretically modeled by a piece-wise constant function with jump discontinuity along a smooth curve, the magnetic edge. Assuming the edge is a closed curve with an axis of symmetry  and the field is sign changing and with exactly two distinct values, we prove that semi-classical tunneling occurs and calculate the magnitude of this tunneling effect.\\

{\it This paper is dedicated to Elliott H. Lieb on the occasion of his 90th birthday.}
\end{abstract}
\maketitle

\section{Introduction}\label{sec:Int}

The purpose of  this paper is to study the magnetic Laplacian in $\mathbb R^2$,
  \begin{equation}\label{eq:P}
\mathcal P_{h}:=(ih\nabla+\mathbf A)^2=\sum_{j=1}^2 (ih \partial_{x_j} +A_j)^2,
\end{equation}
where $\mathbf A:=(A_1,A_2)\in H^1_\textrm{loc}(\mathbb R^2;\mathbb R^2)$, is a magnetic potential generating  the magnetic field $B=\textrm{ curl}\mathbf A:=\partial_{x_1}A_2-\partial_{x_2}A_1\in L^2_\textrm{ loc}(\mathbb R^2;\mathbb R)$.  
We will also discuss the case of the Neumann or Dirichlet realizations of $\mathcal P_h$ in smooth bounded planar domains.

Here $h$ is a positive parameter that tends to $0$, which  can be interpreted as the semi-classical parameter. By writing $h^{-2}\mathcal P_h=(i\nabla+h^{-1}\mathbf A)^2$, we observe that the semi-classical limit, $h\to0_+$, is equivalent to the strong magnetic field limit, $h^{-1}|B|\to+\infty$.

The spectrum of the operator $\mathcal P_h$
has been the subject of  an intense study  in the past decades, particularly in the context  of superconductivity where the magnetic field $B$ is typically a constant function \cite{BS98, Bon05, FH10, HM1, HP03, LuP99}. 

There is an interesting analogy between the results for the Neumann realization of $\mathcal P_h$  in a bounded smooth domain and those for the Schr\"{o}dinger operator, $-h^2\Delta+V$,  with an electric potential $V$, in the full plane. The Schr\"{o}dinger operator was intensively  studied by Helffer--Sj\"{o}strand  \cite{HeSj,HeSj5} and Simon \cite{S},  notably  in the context of quantum tunneling. Bound states of  
$-h^2\Delta+V$ concentrate near the `well' $\Gamma_V:=\{x\in\mathbb R^2\,|\,V(x)=\min_{\mathbb R^2} V\}$; if furthermore $\Gamma_V$ is a regular manifold (i.e. we have a degenerate well), bound states could concentrate near some points of $\Gamma_V$, the `mini-wells'. We have the same picture in the purely magnetic case with a Neumann boundary  condition: bound states  concentrate near the boundary of the domain, whereby the boundary plays the role of a (degenerate) well, and the set of points of maximum curvature plays the role of mini-wells, where bound states decay away from them. Optimal estimates describing the concentration of bound states are very important, since they lead to accurate asymptotics for the low lying eigenvalues. The proof of the decay away from the mini-wells (points of maximum curvature), is more delicate compared to that of the decay  away from the well (boundary).

In this paper, our main focus will be on magnetic fields having a jump-discontinuity.
Magnetic fields that vary on very short scales (nanoscales) have been observed experimentally, see e.g.~\cite{FLBP94}. Their theoretical investigations, in the context of quantum mechanics \cite{PM93, RP00} or graphene \cite{GDMH08}, involve the operator $\mathcal P_h$ but with the magnetic field $B$ being a step function having a discontinuity along a curve, that we will refer to as the \emph{magnetic edge}. 

Earlier rigorous results were devoted to  the case of a flat edge \cite{HPRS16, HS15, I85}.  More recently, non-flat edges have been considered in the context of spectral  asymptotics \cite{A20, AHK} and in the  context  of superconductivity \cite{AKPS19}.  The magnetic edge will play the role of the `well', while  the `mini-wells' are the points  of  maximum curvature of the magnetic edge \cite{AK20}, which is interestingly  in analogy with  the setting of the Neumann realization with a constant magnetic field in a bounded smooth domain.

The case of a single mini-well, where the curvature of the edge  has  a  unique and non-degenerate maximum, was analyzed by Assaad--Helffer--Kachmar  \cite{AHK}. The present paper investigates the situation of a symmetric edge with several mini-wells,  the  simplest case being when there are two non-degenerate maxima of the boundary curvature. We establish a sharp asymptotics of the splitting between the energies of the ground and first excited state, which measures a tunneling effect induced by the geometry of the edge, see Theorem~\ref{thm:FHK} below which is our main result.

Let  us  recall how the   general  strategy of Helffer--Sj\"{o}strand \cite{HeSj, HeSj5} has been applied recently to understand the tunneling effect for the Neumann realization in a  bounded domain
with the breakthrough paper \cite{BHR21} by Bonnaillie-No\"{e}l--H\'erau--Raymond as the crowning achievement.
The first step, already performed in  \cite{HM1} and \cite{FH06}, was the analysis of  a model with a flat boundary  (de\,Gennes model), which yields  localization of bound states near the boundary of the domain (the well), and consequently, leads to a full  asymptotics of the low-lying eigenvalues. The second  step is a formal WKB expansion of bound states \cite{BHR15}. The third step consists of optimal decay estimates of bound states recently achieved in \cite{BHR21}. The importance of this step is  that  it  allows one to rigorously approximate the bound states by the formal WKB expansions,  which eventually paves the way  for the analysis of an interaction matrix whose eigenvalues measure the tunneling effect. The same approach  has been successfully applied in  the context of thin domains \cite{KR17} and the Robin  Laplacian \cite{HK-tams, HKR},  where the proof of the tangential estimates was less technical.

We will follow the same approach outlined above in the case of our discontinuous magnetic field. The model problem with a flat edge was analyzed in \cite{AK20} (see also \cite{AKPS19, HPRS16}), while the full asymptotics for the low lying eigenvalues are obtained in \cite{AHK}. So  we still need WKB expansions  and optimal tangential estimates of bound  states, which we do in the present contribution. Finally, after establishing the  WKB approximation, the
analysis of the interaction matrix   is  quite standard.

 Let us give  an informal statement of our  main result (Theorem~\ref{thm:FHK} below). Suppose that $\Gamma$ is  a smooth, closed curve in $\mathbb R^2$, symmetric with respect to an axis, and with two points of maximum curvature, denoted  by  $s_\ell$  and $s_r$ ($\ell$ refers to ``left'' and $r$ to ``right'', see Fig.~\ref{fig2}).  Consider the magnetic field satisfying $B=1$ in  the interior of $\Gamma$, and $B=a\in(-1,0)$ in the exterior of $\Gamma$.  Under  these assumptions, we  prove that, as $h\to0_+$,  the  spectral gap of the  operator $\mathcal P_h$ is of exponential  order,
\begin{equation}\label{eq:informal-tun}
\lambda_2(h)-\lambda_1(h)\approx \exp\Big(-\frac{\mathsf S^a}{h^{1/4}} \Big),
\end{equation}
where  $\mathsf  S^a$ is the Agmon distance between the  ``wells'' $s_\ell$ and $s_r$ defined by  an appropriate potential that depends on the magnetic field (through the parameter $a$) and the  geometry of $\Gamma$ (through the curvature). 
\begin{figure*}[t]
\includegraphics[width=10cm]{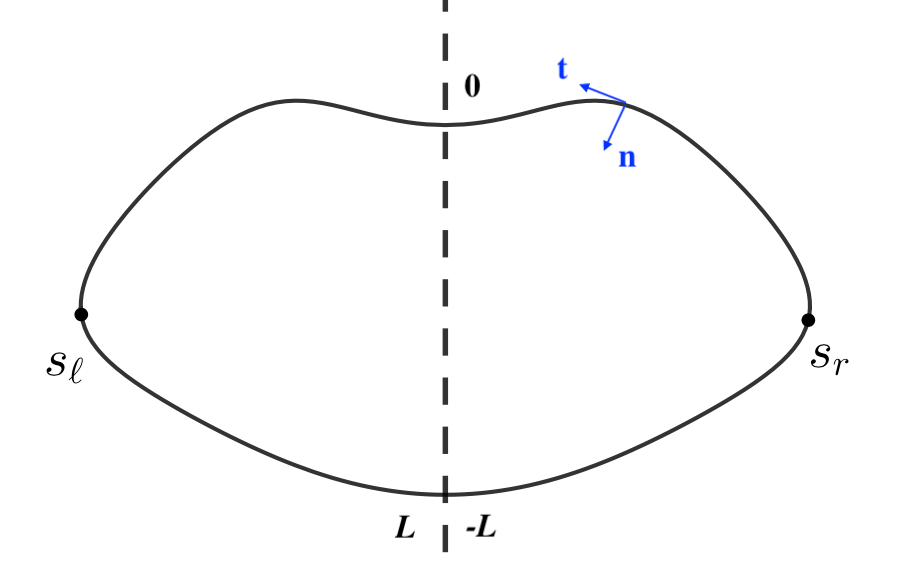}
\caption{A symmetric  domain with respect to  the $y$-axis (dashed line).  The orientation of the boundary is  defined by the direct frame $(\mathbf t,\mathbf n)$, where $\mathbf n$ is the inward normal vector  and $\mathbf t$ is  the unit tangent.  The curvature  along the boundary has two  non-degenerate maxima at the points $a_1$ and $a_2$, with  arc-length coordinates $ s_\ell\in [0,L) $ and $ s_r\in(-L,0]$,   connected by  upward  and downward geodesics  oriented counterclockwise and represented  by $[s_r,0]\cup(0,s_\ell]$ and $[s_\ell,L]\cup(-L,s_r]$  respectively.  These upward and downward geodesics will be denoted by $[s_r,s_\ell]$ and $[s_\ell,s_r]$ respectively.}\label{fig2}
\end{figure*}
The  asymptotics in \eqref{eq:informal-tun} (more precisely that in Theorem~\ref{thm:FHK})  is a consequence of  quantum tunneling.  It is important to note  that it is induced purely by the magnetic field, thereby providing an
example of  a  purely magnetic quantum tunneling---where the case of \cite{BHR21} also required the interaction with the boundary.  If  we look at earlier results on  the tunneling effect, with or without  magnetic field,  we observe that the tunneling is induced by an external potential \cite{HeSj, FSW} or by confinement to a bounded/thin domain \cite{HKR, KR17, BHR}.  For the Neumann realization of $\mathcal P_h$,  the presence of  the magnetic field adds a challenging difficulty in the estimate of  the magnitude of the  tunneling that was recently solved in \cite{BHR21}. Our proof of \eqref{eq:informal-tun} is very close to that of \cite{BHR21}, but it relies on new elements that follow from a deep investigation of magnetic steps \cite{AK20, AHK}.

Let  us  give   some of the heuristics  behind  the computations leading to  \eqref{eq:informal-tun}. We can construct two quasi-modes having  the following structure
\begin{align*}
\Psi_{h,\ell}(s,t)&\approx e^{i\theta_{h,\ell}(s)}\,e^{-\Phi_\ell(s)/h^{1/4}}\,f_{0,\ell}(s)\phi_a(h^{1/2}t),\\ \Psi_{h,r}(s,t)&\approx e^{i\theta_{h,r}(s)}e^{-\Phi_r(s)/h^{1/4}}\,f_{0,r}(s)\phi_a(h^{1/2}t),
\end{align*}
where
$s$ denotes the arc-length parameter along $\Gamma$, $t$ denotes the normal distance to $\Gamma$  with the convention that $t>0$  in  the interior of $\Gamma$  and  $t<0$ in the exterior of $\Gamma$. The functions  $\Phi_\ell$ and $\Phi_r$ are  non-negative and satisfy
$\Phi_{\ell}(s_\ell)=0$ and $\Phi_r(s_r)=0$, so that  $\Psi_{h,\ell}$ (resp.  $\Psi_{h,r}$) is localized near $s_\ell$  (resp. near $s_r$).  The phase functions  $\theta_{h,\ell}$ and $\theta_{h,r}$ involve  the topology of the discontinuity curve and  a spectral constant.  The function $\phi_a$ is the ground state eigenfunction of   a model operator  related to the discontinuity of the magnetic field (see Sec.~\ref{sec:ms}).  Finally,  $f_{0,\ell}$ and $f_{0,r}$ are solutions of  appropriate transport equations (see Theorem~\ref{thm:WKB}).

 Up to  truncation,  the quasi-modes $\Psi_{h,\ell}$ and  $\Psi_{h,r}$ are approximations of actual bound states 
 \begin{equation}\label{eq:app-introduction}
g_{h,\ell}(s,t)\approx  \Psi_{h,\ell}(s,t),\quad g_{h,r}(s,t) \approx  \Psi_{h,r}(s,t),
\end{equation}
where the bound states $g_{h,\ell}$ and $g_{h,r}$ are defined via the orthogonal  projection $\Pi$ on $ V:=\oplus_{i=1}^2\textrm{ Ker}\left(\mathcal P_h-\lambda_i(h)\right)$ as follows
\[
g_{h,\ell}(s,t):=\Pi\Psi_{h,\ell}(s,t),\quad
g_{h,r}(s,t):=\Pi\Psi_{h,r}(s,t).\]
By the Gram-Schmidt process, we transform $\{g_{h,\ell},g_{h,r} \}$  to an orthonormal basis $\mathcal B$ of $V$ and we denote  by $\mathsf  M_h$ the matrix   relative to $\mathcal B$ of the restriction of $\mathcal P_h$ to $V$.   The spectral  gap for the operator $\mathcal P_h$ is the same as that for the matrix $\mathsf M_h$,
\[\lambda_2(h)-\lambda_1(h)= \lambda_2(\mathsf M_h)-\lambda_1(\mathsf M_h) \,.\]
Using the approximation in  \eqref{eq:app-introduction}, we get an approximate matrix $\widehat{\mathsf M}_h$ of  $\mathsf M_h$ whose  spectral gap can be  explicitly estimated (compare with  \eqref{eq:informal-tun})
\[ \lambda_2(\widehat{\mathsf M}_h)-\lambda_1(\widehat{\mathsf M}_h)\approx \exp\Big(-\frac{\mathsf S^a}{h^{1/4}} \Big)\,.\]
We have then to  show that the spectral  gap for the  matrix $\widehat{\mathsf M}_h$  is a good approximation of that of $\mathcal P_h$ up to an appropriate remainder, more precisely
\[ \lambda_2(\mathsf M_h)-\lambda_1(\mathsf M_h)= \big(\lambda_2(\widehat{\mathsf M}_h)-\lambda_1(\widehat{\mathsf M}_h)\big)\big(1+o(1)\big)\,.\]
Such an estimate is closely related to optimal decay estimates of bound states (resp. approximate bound states) of the operator $\mathcal P_h$, which yield  accurate errors for the approximation in \eqref{eq:app-introduction}.
\subsection*{Organisation}
 In Section~\ref{sec:B=cst}, we review  the recent result of \cite{BHR21} for the  Neumann magnetic Laplacian with  a  constant magnetic field, and introduce the related de\,Gennes model for a flat boundary. In  Section~\ref{sec:ms}, we introduce the magnetic edge along with  the related flat  edge model, and state our main result, Theorem~\ref{thm:FHK}, for the operator $\mathcal P_h$ with  a magnetic step. In Section~\ref{sec:curv}, we express $\mathcal P_h$ in a Fr\'enet frame and reduce the spectral analysis to an  operator defined near the edge. In Section~\ref{sec:WKB}, we introduce operators with a single well (with ground states localized near a single point of maximum curvature), and perform a WKB expansion for an approximate ground state (see Theorem~\ref{thm:WKB}). In Section~\ref{sec:dec}, we explain how optimal tangential estimates can be derived along the lines of the proof of the similar statement in \cite{BHR21}. Finally, in Section~\ref{sec:IM}, we introduce the interaction matrix and  finish the proof of Theorem~\ref{thm:FHK}, by referring to \cite{BHR21} for the detailed computations,  which are essentially  the same in  our setting. 
\section{Uniform magnetic fields}\label{sec:B=cst}
In this section, we review some results on the Neumann  realization of the  operator $\mathcal P_h$ with a constant magnetic field. We assume that 
\begin{equation}\label{eq:Sigma}
\left\{\,\begin{aligned}
&\Omega\subset\mathbb R^2\text{ is a  }\text{simply}\text{ connected open set,}\\
&\Sigma:=\partial\Omega\text{ is a }C^\infty\text{ smooth closed curve.}\\
\end{aligned}\,\right\}
\end{equation}
and
\begin{equation}\label{eq:B=1}
\mathbf A=\frac12(-x_2,x_1)\quad\textrm{ and}\quad B=\textrm{ curl}\,\mathbf A\equiv1\,,
\end{equation}
We consider  $\mathcal P_h$  introduced in \eqref{eq:P},  as a self-adjoint operator in  $L^2(\Omega)$, with domain,
\begin{equation*}
\textsf{ Dom}(\mathcal P_{h})=\{u\in H^2(\Omega)~|~{\mathbf  n}\cdot(h\nabla-i\mathbf A)u|_{\partial \Omega}=0 \},
\end{equation*}
where $H^2(\Omega)$ denotes the Sobolev space $W^{2,2}(\Omega)$, and $\mathbf n$ the unit normal vector of  $\Sigma$,  pointing inwards with respect  to $\Omega$.
\subsection{Full asymptotics and decay of bound states}
 The conditions in \eqref{eq:Sigma} ensure that $\Omega$ is bounded and that $\mathcal P_h$ has compact resolvent. Let $(\lambda_n(h))_{n\geq 1}$ be the sequence of eigenvalues of $\mathcal P_h$.  
In generic  situations, that will  be explained precisely later on, there exist complete expansions of the eigenvalues of $\mathcal P_h$, in the form \cite{FH06},
\begin{equation}\label{eq:lambdan}
\lambda_n(h)\sim\Theta_0h -k_{\max}C_1h^\frac 32 +C_1\Theta_0^{\frac 14}(2n-1)\sqrt{-\frac 32  k_2}h^\frac 74+\sum_{j\geq 15}\zeta_{j,n} h^{j/8}\,.
\end{equation}
The coefficients $\Theta_0$ and $C_1$ appearing in \eqref{eq:lambdan} are universal positive constants related to  the de\,Gennes model in the half-plane (see Sec.~\ref{sec:dG}).  The coefficients $k_{\max}$
and  $k_2$ are related to  the curvature on  the boundary. Let $\Sigma$ be parameterized by arc-length $s$ and  denote by $k(s)$  the curvature of $\Sigma$ at $s$ (see Sec.~\ref{sec:curv} for the precise definition of $k$; in particular  the orientation is chosen so that $k\geq 0$ if $\Omega$ is convex). The asymptotics in \eqref{eq:lambdan} holds provided  the curvature $k$ attains its maximum value non-degenerately and at a unique point, i.e. 
\begin{equation}\label{eq:hyp-gd}
k_{\max}:=\max_{\Sigma} k(s)=k(0)\quad\textrm{ with}\quad  k_2:=k''(0)<0\,.
\end{equation}
The sequence $(\zeta_{j,n})_{j\geq 15}$ is constructed recursively, and it can be shown that  $\zeta_{j,1}=0$ for odd  $j$ \cite{BHR15}.

The  derivation of \eqref{eq:lambdan} is related to the decay of bound states. Assume that $n$ is fixed and for all $h>0$ that $u_{h,n}$ is an eigenfunction of $\mathcal P_h$, normalized in $L^2(\Omega)$ and  with eigenvalue $\lambda_n(h)$. There  exist constants $\alpha_1,C_{1,n}>0$ such that
\[ \int_\Omega |u_{h,n}|^2\exp\Big(\frac{\alpha_1 \textrm{ dist}(x,\Sigma)}{h^{1/2}}\Big)dx\leq C_{1,n}\,.\]
This estimate says that the bound state $u_{h,n}$ concentrates near the boundary $\Sigma$ and is valid even when \eqref{eq:hyp-gd} is not satisfied \cite{HM1}. If moreover \eqref{eq:hyp-gd} holds, then $u_{h,n}$ concentrates near the point of maximal curvature as follows: There exist constants $\epsilon_0,\alpha_2,C_{2,n}>0$ such that \cite{FH06}
\begin{equation}\label{eq:dec-uh}
\int_{\textrm{ dist}(x,\Sigma)<\epsilon_0}|u_{h,n}|^2\exp\Big(\frac{\alpha_2|s(x)|^2}{h^{1/4}}\Big)dx\leq C_{2,n}\,,
\end{equation}
where $s(x)$ denotes the arc-length coordinate of the point $p(x)\in\partial\Omega$ defined by 
$\textrm{dist}(x,p(x))=\textrm{ dist}(x,\Sigma)$. The decay estimate \eqref{eq:dec-uh} is a key ingredient in the derivation of  the  asymptotics in \eqref{eq:lambdan}, but is not sufficient to handle the case of symmetries that we shall discuss below.
 
Let us  examine    the case  where the curvature attains its maximum at several points $s_1,\cdots,s_N$. For all $j\in\{1,\cdots,N\}$  and $m\in\mathbb N$, we introduce
\[\lambda_{m,j}^\textrm{ app}(h)=\Theta_0h -k_{\max}C_1h^\frac 32 +C_1\Theta_0^{\frac 14}(2m-1)\sqrt{-\frac 32  k''(s_j)}\,h^\frac 74\,. \]
Consider a relabeling  $(m_n,j_n)_{n\geq 1}$ of $(m,j)_{m\geq 1,1\leq j\leq N}$ such that
\[\lambda_{m_1,j_1}^{\textrm{app}}\leq \lambda_{m_2,j_2}^\textrm{app}\leq \cdots \,.\]
Then, \eqref{eq:lambdan} is replaced with
\begin{equation}\label{eq:mult-well}
\lambda_n(h)=\lambda^\textrm{app}_{m_n,j_n}+o(h^{\frac74})\,. 
\end{equation}
If  additionally $k''(s_{j_1})=k''(s_{j_2})$,  then  $\lambda_2(h)-\lambda_1(h)=o(h^{\frac{7}4})$ and we loose the information on the simplicity of the eigenvalues. 
Consequently, we need a more detailed analysis in the  case of symmetries, which  will rely on an optimal tangential decay estimate improving the one given   in \eqref{eq:dec-uh}.  We will discuss  these decay estimates later in Sec.~\ref{sec:dec}. Our next step is the review of an important model with a flat boundary.

\subsection{The de\,Gennes model: flat boundary}\label{sec:dG}

The analysis of the model case where $\Omega=\mathbb R\times\mathbb R_+$ and $B=\textrm{ curl\,}\mathbf A=1$   leads us naturally to the family (parametrized by $\xi\in \mathbb  R$) of harmonic oscillators (de\,Gennes model)
\begin{equation}
\mathfrak h^{N}[\xi]=-\frac{d^2}{d\tau^2}+(\xi+\tau)^2,
\end{equation}
on the semi-axis  $\mathbb R_+$ with Neumann boundary condition at $\tau=0$. Let us denote by $(\mu_j^{N}(\xi))_{j\geq 1}$ the sequence of eigenvalues of $\mathfrak h^N[\xi]$. The de\,Gennes constant is then defined as follows
\begin{equation}
\Theta_0=\inf_{\xi\in\mathbb R}\mu^{N}_1(\xi)\,.
\end{equation}
There exists a unique minimum $\xi_0<0$ such that
\[ \Theta_0=\mu^{N}_1(\xi_0)\,.\]
Furthermore, $\xi_0=-\sqrt{\Theta_0}$, $(\mu^N_{1})''(\xi)>0$ and $\frac12<\Theta_0<1$.   Denoting by  $u_0$  the positive and normalized ground state of $\mathfrak h^N[\xi_0]$, we can introduce the constant $C_1$ appearing in \eqref{eq:lambdan},
\begin{equation}
C_1=\frac{|u_0(0)|^2}{3}\,.
\end{equation}

\subsection{Symmetric domains and tunneling}\label{sec:SymDom}
We  continue to work under the conditions in \eqref{eq:Sigma} but  we assume furthermore that the domain $\Omega$  is symmetric with respect to an axis and the curvature of its boundary $\Gamma$ has exactly two non-degenerate maxima. More precisely, the hypotheses are (see Fig~\ref{fig2}):
\begin{assumption}\label{ass:symN}~
\begin{enumerate}[\rmfamily i)]
	\item $\Omega$ is symmetric with respect to the $y$-axis.
	\item The curvature $k$ on  $\Sigma$ attains its maximum at exactly two symmetric points $a_1=(a_{1,1},a_{1,2})$ and $a_2=(a_{2,1},a_{2,2})$ with $a_{1,1}<0$ and  $ a_{2,1}>0$.
	\item  Denoting by $s_r$ and $s_\ell$ the arc-length  coordinates of $a_1$ and $a_2$ respectively, we have $k''(s_r)=k''(s_\ell)<0$.
\end{enumerate}
\end{assumption}
This situation induces a tunneling effect where the energy difference between the ground and first excited states is exponentially small. The magnitude of this splitting  has been  rigorously computed  recently in \cite{BHR21}. 

Let us introduce the  following effective quantities:
\begin{equation}\label{eq:eff-pot}
V(s)=\frac{2C_1(k_{\max}-k(s))}{(\mu_1^N)''(\xi_0)}\,,
\end{equation}
and
\begin{equation}\label{eq.Aud}
\begin{aligned}
\mathsf{A}_{\mathsf{u}}&=\exp\left(-\int_{[s_{r}, 0]} \frac{ (V^\frac 12 )' (s)+g}{ \sqrt{V(s)}} ds\right)\,,\\
\mathsf{A}_{\mathsf{d}}&=\exp\left(-\int_{[s_{\mathsf{\ell}}, L]} \frac{ (V^\frac 12 )' (s) -g}{ \sqrt{V(s)}} ds\right)\,,\\
g&=\left(V''(s_{r})/2\right)^\frac 12=\left(V''(s_{\mathsf{\ell}})/2\right)^\frac 12\,.
\end{aligned}
\end{equation}
In  the above formulae, $0$  and $L$ are the arc-length coordinates of the points of intersection 
between the  $y$-axis and the curve $\Sigma$, with the convention that $0$ represents  the point on the  upper  part of $\Sigma$ (see Fig.~\ref{fig2}).
\begin{theorem}[Bonnaillie-No\"{e}l--H\'erau--Raymond \cite{BHR21}]\label{thm:BHR}~
Suppose that  \eqref{eq:Sigma},  \eqref{eq:B=1} and Assumption~\ref{ass:symN} hold. Then the first and second eigenvalues of  $\mathcal P_h$ satisfy, as $h\to0_+$,
\[\lambda_2(h)-\lambda_1(h)=2| w(h)|+o(h^{\frac{13}{8}}e^{-\mathsf{S}/h^{\frac 14}})\,,\]
where
\begin{align*}
w(h)&=( \mu_1^N)''(\xi_0) h^{\frac{13}{8}} \pi^{-\frac 12} g^{\frac12}\\
&\quad \times
\left(\mathsf{A}_{\mathsf{u}} \sqrt{V(0)}e^{- \mathsf{S}_{\mathsf{u}}/h^{1/4}} e^{iL f(h)}+\mathsf{A}_{\mathsf{d}} \sqrt{V(L)}e^{- \mathsf{S}_{\mathsf{d}}/h^{1/4}} e^{-iLf(h)}\right)\,,
 \end{align*}
 and
\begin{enumerate}[\rmfamily i.]
\item The potential $V$ is introduced in \eqref{eq:eff-pot};
\item  $\mathsf S$ is  the Agmon distance between the wells,
\begin{equation}\label{defS}
\mathsf{S} =\min \left(\mathsf{S}_{\mathsf{u}},\mathsf{S}_{\mathsf{d}}\right),~  \mathsf{S}_{\mathsf{u}}=\int_{[s_{r},s_{\mathsf{\ell}}] } \sqrt{V(s)} \,ds,~  \mathsf{S}_{\mathsf{d}}=\int_{[s_{\mathsf{\ell}}, s_{\mathsf{r}}] } \sqrt{V(s)} \,ds\,;
\end{equation}
\item $\mathsf{A}_{\mathsf{u}}$, $\mathsf{A}_{\mathsf{d}}$ and $g$ are defined in \eqref{eq.Aud};
\item $f(h)=\gamma_0/h+\xi_0/h^{1/2}-\alpha_0$ with
\[\gamma_0=\frac{|\Omega|}{|\Sigma|},\]
where $|\Sigma|$ is the length of $\Sigma$, and $\alpha_0$ is a constant dependent on $\Omega$.
\end{enumerate}
\end{theorem}
 Theorem~\ref{thm:BHR} can be extended to the situation of $N\geq 3$ wells, which corresponds to a domain having symmetry by rotation of angle $2\pi/N$ and $N$ points of maximum curvature (see Fig.~\ref{fig:star}).
\begin{figure*}[t]
\includegraphics[width=5cm]{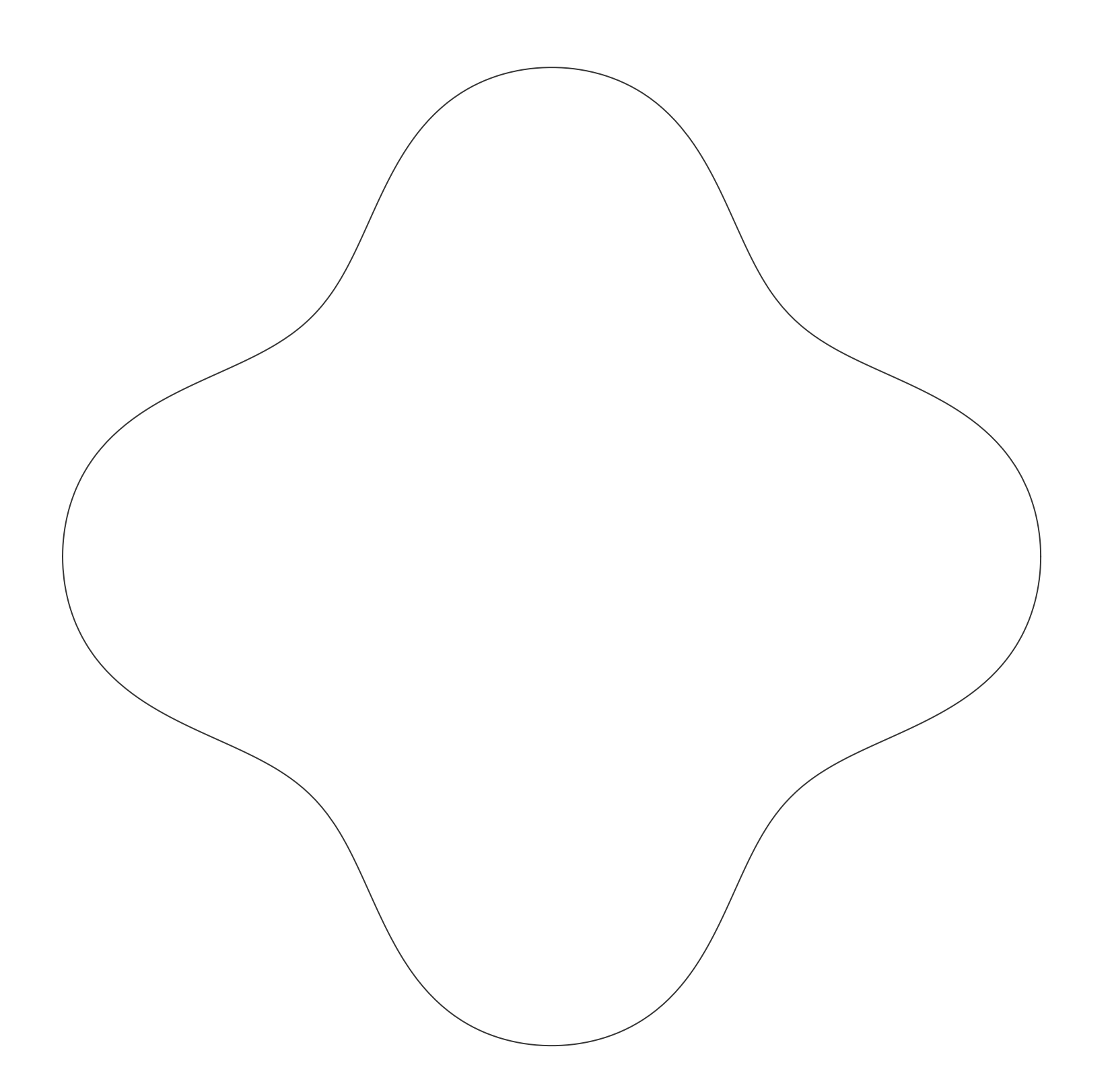}
\caption{A symmetric  domain with respect to  the origin with $N=4$ points of maximum curvature.}\label{fig:star}
\end{figure*}
\section{Magnetic steps}\label{sec:ms}

 The tunneling effect in Theorem~\ref{thm:BHR} is a consequence of  the magnetic field and imposing the Neumann  boundary condition (if a magnetic field were not present, the first eigenvalue would be simple and equal to $0$, while the Neumann boundary condition inforces bound states to concentrate near  the boundary  points of maximum curvature thereby inducing a phenomenon of \emph{multiple wells}).
 
 The present contribution is concerned with the following question:\\ \textit{Can we observe a tunneling  effect, similar to the one in Theorem~\ref{thm:BHR},  but induced purely by the magnetic field ?}\\ 
 That is, we would like to construct  an example where the tunneling is not a consequence of imposing a boundary condition, but rather a consequence of  the nature of  the magnetic field.
We will give an  affirmative answer  by working in the full plane $\mathbb R^2$ and considering a magnetic field with a discontinuity  along a smooth curve\footnote{ From a technical perspective, the magnetic discontinuity curve plays the same role in our case as the boundary does in Theorem~\ref{thm:BHR}.}  (the magnetic edge). 
In  the case of a flat  edge, we get  a  model in the  full plane  which  plays the role of  the de\,Gennes model for uniform magnetic fields. When the edge is non-flat and has symmetries, we observe an interesting  tunneling effect.
\subsection{A new model: flat  edge}\label{sec:FlatEdge}
Let us recall the model  in $\mathbb R^2$ where $B=\textrm{ curl\,}\mathbf A=\mathbf 1_{\mathbb R_+\times\mathbb R }+a\mathbf 1_{\mathbb R_-\times\mathbb R}$ and $a\in[-1,0)$ is  a  fixed constant\footnote{It is important for us to have $a<0$, because in the opposite case, $a \in (0,1)$, $\mu_a(\xi)$ defined in \eqref{mu_a_1} becomes a monotone increasing function with $\inf_{\xi \in{\mathbb R}} \mu_a(\xi) = a$.  This implies that the magnetic step will no longer attract the ground state, i.e. we do not expect localization near the magnetic step in this case.
}. We get in this case a family of Schr\"{o}dinger operators \cite{HPRS16}
\begin{equation}\label{eq:ha}
\mathfrak h_a[\xi]=-\frac{d^2}{d\tau^2}+V_a(\xi,\tau),
\end{equation}
on $L^2(\mathbb R)$, where $\xi\in\mathbb R$ is a parameter and 
\begin{equation}\label{eq:potential}
V_a(\xi,\tau)=\big(\xi+b_a(\tau)\tau\big)^2,\quad b_a(\tau)=\mathbf{1}_{\mathbb R_+}(\tau)+a\mathbf{1}_{\mathbb R_-}(\tau)\,.
\end{equation}
We introduce the ground state energy   of $\mathfrak h_a[\xi]$, 
\begin{equation}\label{mu_a_1}
\mu_a(\xi)=\inf_{u\in B^1(\mathbb R),u\neq0} \frac{
\|u'\|_{L^2(\R)}^2+\|\sqrt{V_a}\, u\|^2_{L^2(\R)}}
{\|u\|^2_{L^2(\mathbb R)}}\,,
\end{equation}
along with the following constant 
\begin{equation}\label{eq:beta}
\beta_a:=\inf_{\xi \in \mathbb R} \mu_a(\xi)=\mu_a(\zeta_a)\,,
\end{equation}
where $\zeta_a<0$, is the unique minimum of $\mu_a(\cdot)$.  Let $\phi_a$ be the \emph{positive} and $L^2$-normalized ground state of $\mathfrak  h_a[\zeta_a]$. We have \cite{AK20}
\begin{equation}\label{eq:c2}
c_2(a):=\frac12\mu_a''(\zeta_a)>0
\end{equation}
and 
\begin{equation}\label{eq:beta*}
 |a|\Theta_0< \beta_a<\min(|a|,\Theta_0),\quad \phi_a'(0)<0\quad (-1<a<0)\,. 
 \end{equation}
For $a=-1$, we have    by  a symmetry argument
\begin{equation}\label{eq:deG}
\beta_{-1} = \Theta_0\,,\quad \zeta_{-1}=\xi_0\,,\quad  \phi_{-1}(\tau)=u_0(|\tau|)\,,
\end{equation}
thereby returning to the  de\,Gennes model introduced in Sec.~\ref{sec:dG}.

Later on, the following negative constant will be of particular interest,
\begin{equation}\label{eq:m3}
M_3(a)=\frac 13\Big(\frac 1a-1\Big)\zeta_a\phi_a(0)\phi_a'(0) <0.
\end{equation}

\subsection{Curved edge and single well}\label{sec:Edge}
We return to the operator $\mathcal P_h$ in \eqref{eq:P}.  Here and in the  rest of the paper, we will work under the  following assumption\footnote{Our results are likely to hold when  $\Gamma$ is $C^N$ smooth for some integer $N\geq 1$.  We  impose  the $C^\infty$  hypothesis since we use psudo-differential calculus and sought errors of order $\mathcal O(h^\infty)$.}
\begin{equation}\label{eq:Gam}
\left\{\,\begin{aligned}
&\Omega_1\subset\mathbb R^2\text{ is a }\text{simply connected open set,}~\Omega_2=\mathbb R^2\setminus\overline{\Omega}_1,\\
&\Gamma:=\partial\Omega_1 \text{ is a }C^\infty\text{ smooth closed curve.}\\
\end{aligned}\,\right\}
\end{equation}
and that  the magnetic field is a step function (see Fig.~\ref{fig1})
\begin{figure*}[t]
\includegraphics[width=6cm]{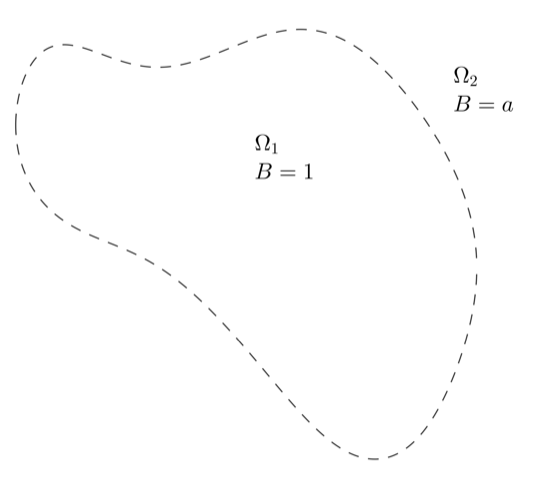}
\caption{The plane  $\mathbb R^2=\Omega_1\cup\Omega_2\cup\Gamma$ with  the  non symmetric edge  $\Gamma=\partial\Omega_1$ dashed.}\label{fig1}
\end{figure*}
\begin{equation}\label{eq:B-ms}
B=\mathbf 1_{\Omega_1}+a\mathbf 1_{\Omega_2}\quad\textrm{ where~}-1< a<0\,.
\end{equation}
The operator  $\mathcal P_h$ is then self-adjoint in $L^2(\mathbb  R^2)$ with domain\footnote{Since $\mathbf A\in H^1_{\rm loc}(\R^2)$,  there is no jump across  $\Gamma$ of $u$ and $\mathbf n\cdot(h\nabla-i\mathbf  A)u$,  $\forall u\in{\rm Dom}(\mathcal P_h)$. } 
\begin{equation}\label{eq:DomP}
\textsf{ Dom}(\mathcal P_{h})=\{u\in L^2(\mathbb R^2)~:~(h\nabla-i\mathbf A)^{j}u\in  L^2(\mathbb R^2),~j=1,2\}.
\end{equation}
By Persson's lemma \cite{P}, the essential spectrum  of $\mathcal P_h$ is determined by the magnetic field at infinity (in our case it is equal to $a$),  so  
\[ \inf \sigma_\textrm{ ess}(\mathcal P_h)=|a| h\,.\]
Since $\beta_a<\min(|a|,\Theta_0)$, bound states of $\mathcal P_h$  are localized near the edge \cite{AK20}. More precisely,  for every $n\in\mathbb N$, there exist constants  $\alpha,h_0,C_n>0$ such that,
\begin{equation}\label{eq:dec-norm}
\int_{\mathbb R^2}\big(|u_{h,n}|^2+h^{-1}|(h\nabla-i\mathbf A)u_{h,n}|^2 \big)\exp\Big(\frac{\alpha\, \textrm{  dist}(x,\Gamma)}{h^{1/2}} \Big)dx\leq C_n,
\end{equation}
for all $h\in (0,h_0]$, where $u_{h,n}$ is a normalized eigenfunction associated to the $n$'th eigenvalue of $\mathcal  P_h$.

\begin{remark}[The case of bounded domains]
We can also consider the Dirichlet or Neumann realizations of $\mathcal P_h$ in a bounded smooth domain $\Omega$, in which case the spectrum is purely discrete.
Related to our setting is \cite[Thm.~1.2]{AHK} dealing with a somehow different geometric condition, where the operator  $\mathcal  P_h$ is considered in $L^2(\Omega)$ with  Dirichlet boundary condition, $\Omega_1\subset \Omega$ and $\Gamma$ a  smooth curve that meets $\partial\Omega$ transversely, see Fig.~\ref{fig3.2}. However, the proofs are not altered by considering the new setting of $\mathcal P_h$ above ($\mathcal P_h$ in the full plane and closed curve  $\Gamma$).  The main reason is that the property  $\beta_a<|a|$  for $-1<a<0 $ ensures   the localization of the bound states   near the edge $\Gamma$.
\end{remark}

So the following  result essentially follows from \cite[Thm.~1.2]{AHK}:
\begin{theorem}\label{thm:AHK}
Assume that \eqref{eq:Gam} and \eqref{eq:B-ms} hold and that the curvature $k$ of $\Gamma$ has  a unique non-degenerate  maximum, i.e.
\[k_{\max}:=\max_{\Gamma} k(s)=k(0)\quad\textrm{ with}\quad  k_2:=k''(0)<0\]
Then, for all $n\in\mathbb N^*$  the $n$-th eigenvalue $\lambda_n(h)$ of $\mathcal P_{h}$,  defined in~\eqref{eq:P}, satisfies as $h\rightarrow 0$, 
\[\lambda_n(h)= \beta_ah+k_{\max} M_3(a)h^\frac 32+(2n-1)\sqrt{\frac{{ k_2}M_3(a)c_2(a)}{2}}h^\frac 74+\mathcal O(h^{\frac{15}8}),\]
where $\beta_a$,   $c_2(a)$  and $M_3(a)$ are  introduced in \eqref{eq:beta},  \eqref{eq:c2} and  \eqref{eq:m3} respectively.
\end{theorem}
Looking more closely at Theorem~\ref{thm:AHK}, we observe that the third  term in  the expansion of $\lambda_n(h)$ is effectively given (up to the factor of $h^{3/2}$) by the $n$-th eigenvalue of  the  following 1D operator  on $L^2\big(\mathbb R /(2L\mathbb Z)\big)$,
\begin{equation}\label{eq:eff-op}
\mathfrak L_h^\textrm{ eff}=\frac{\mu_a''(\zeta_a)}{2}\left(-h^{\frac 12} \partial_s^2+V_a(s)\right)\,,\quad V_a(s)=\frac{2M_3(a)(k(s)-k_{\max})}{\mu''_a(\zeta_a)},
\end{equation}
where $L=|\Gamma|/2$ and $|\Gamma|$  denotes the arc-length of $\Gamma$.  
Notice, that $V_a\geq 0$, due to the sign of $M_3(a)$ (see \eqref{eq:m3}).
This point of view is important in order to discuss the case  where $\Gamma$ has symmetries and the splitting of the  eigenvalues is no more of  fractional order in $h$.

In the presence of  several points of maximal curvature, a variant of Theorem~\ref{thm:AHK} continues to hold but we may loose the information on the simplicity of the eigenvalues, exactly in the same manner observed for the Neumann problem (see \eqref{eq:mult-well}).

\subsection{Symmetric edge and tunneling}\label{sec:SymEdge}

Suppose that,   in addition to \eqref{eq:Gam} and  \eqref{eq:B-ms},  the following holds (see Fig~\ref{fig2}): 
\begin{assumption}\label{ass:sym}~
\begin{enumerate}[\rmfamily i)]
	\item $\Omega_1$ is symmetric with respect to the $y$-axis.
	\item The curvature $k$ on  $\Gamma$ attains its maximum at exactly two symmetric points $a_1=(a_{1,1},a_{1,2})$ and $a_2=(a_{2,1},a_{2,2})$ with $ a_{1,1}<0$ and  $ a_{2,1}>0$.
	\item  Denoting by $s_r$ and $s_\ell$ the arc-length  coordinates of $a_1$ and $a_2$ respectively, we have $k''(s_r)=k''(s_\ell)<0$.
\end{enumerate}
\end{assumption}
This is  exactly the same geometric assumption on $\Omega$ as Assumption~\ref{ass:symN}  for the Neumann realization in $L^2(\Omega)$, with the edge $\Gamma$ playing the  role of  $\Sigma$, the boundary of $\Omega$.

The presence of a symmetric edge yields a symmetric potential, and  consequently two wells, in  the effective operator introduced in \eqref{eq:eff-op},  which in turn will induce a tunneling effect whose order of magnitude can be measured   by the following quantities (similarly to what we have seen in Theorem~\ref{thm:BHR}):
\begin{equation}\label{eq.Aud-a}
\begin{aligned}
\mathsf{A}_{\mathsf{u}}^a&=\exp\left(-\int_{[s_{r}, 0]} \frac{ (V_a^\frac 12 )' (s)+g_a}{ \sqrt{V_a(s)}} ds\right)\,,\\
\mathsf{A}_{\mathsf{d}}^a&=\exp\left(-\int_{[s_{\mathsf{\ell}}, L]} \frac{ (V_a^\frac 12 )' (s) -g_a}{ \sqrt{V(s)}} ds\right)\,,\\
g_a&=\left(V_a''(s_{r})/2\right)^\frac 12=\left(V_a''(s_{\mathsf{\ell}})/2\right)^\frac 12\,.
\end{aligned}
\end{equation}
Up to leading order, the operator  in \eqref{eq:eff-op} continues to be effective  under the new assumptions on the edge, modulo additional terms related to  the circulation of the magnetic field   and the geometry.   
\begin{theorem}\label{thm:FHK}
Suppose that  Assumption~\ref{ass:sym} holds in addition to \eqref{eq:Gam} and \eqref{eq:B-ms}. The first and second eigenvalues of  $\mathcal P_h$ satisfy as $h\to0_+$,
\[\lambda_2(h)-\lambda_1(h)=2|w_a(h)|+o(h^{\frac{13}{8}}e^{-\mathsf{S}^a/h^{\frac 14}})\,,\]
where: 
\begin{align*}
 w_a(h)&= \mu_a''(\zeta_a) h^{\frac{13}{8}} \pi^{-\frac 12} g^{\frac12}_a \\
&\quad\times
\left(\mathsf{A}_{\mathsf{u}}^a \sqrt{V_a(0)}e^{- \mathsf{S}_{\mathsf{u}}^a/h^{1/4}} e^{iL f_a(h)}+\mathsf{A}_{\mathsf{d}}^a \sqrt{V_a(L)}e^{- \mathsf{S}_{\mathsf{d}}^a/h^{1/4}} e^{-iLf_a(h)}\right)\,,
\end{align*}
with $\mu_a$ and $\zeta_a$ introduced in Section \ref{sec:FlatEdge},
 and
\begin{enumerate}[\rmfamily i.]
\item The potential $V_a$ is introduced in \eqref{eq:eff-op};
\item  $\mathsf S^a$ is  the Agmon distance between the wells,
\begin{equation}\label{defSa}
\mathsf{S}^a =\min \left(\mathsf{S}_{\mathsf{u}}^a,\mathsf{S}_{\mathsf{d}}^a\right),~  \mathsf{S}_{\mathsf{u}}^a=\int_{[s_{r},s_{\mathsf{\ell}}] } \sqrt{V_a(s)} \,ds,~  \mathsf{S}_{\mathsf{d}}^a=\int_{[s_{\mathsf{\ell}}, s_{\mathsf{r}}] } \sqrt{V_a(s)} \,ds\,;
\end{equation}
\item $\mathsf{A}_{\mathsf{u}}^a$, $\mathsf{A}_{\mathsf{d}}^a$ and $g_a$ are defined in \eqref{eq.Aud-a};
\item $f_a(h)=\gamma_0/h+\zeta_a/h^{1/2}-\alpha_0(a)$ with 
\begin{equation}\label{eq:circ}
\gamma_0=\frac{|\Omega_1|}{|\Gamma|},
\end{equation}
and $\alpha_0(a)$ is a constant dependent on $a$ and $\Omega_1$.
\end{enumerate}
\end{theorem}
Theorem~\ref{thm:FHK} is the analogue of Theorem~\ref{thm:BHR} but for the situation where tunneling is due to the  discontinuity of the magnetic field (without the need for imposing a Neumann boundary condition). 
As in the proof of Theorem~\ref{thm:BHR} in \cite{BHR21}, the proof of Theorem~\ref{thm:FHK} relies on an optimal  tangential decay estimate of  ground states.

\subsection{Bounded domains}\label{sec:extensions}
Theorem~\ref{thm:FHK} continues to hold if we consider the Dirichlet or Neumann realization of the operator $\mathcal P_h$ in $L^2(\Omega)$, where $\Omega$ is a domain with a $C^2$ boundary such that $\overline{\Omega}_1\subset\Omega$ (see Fig~\ref{fig3.1}). Thanks to \eqref{eq:beta*}, bound states of $\mathcal P_h$ are  localized near $\Gamma=\partial\Omega_1$, and the proof of Theorem~\ref{thm:FHK} is not altered.
\begin{figure*}[t]
\includegraphics[width=7cm]{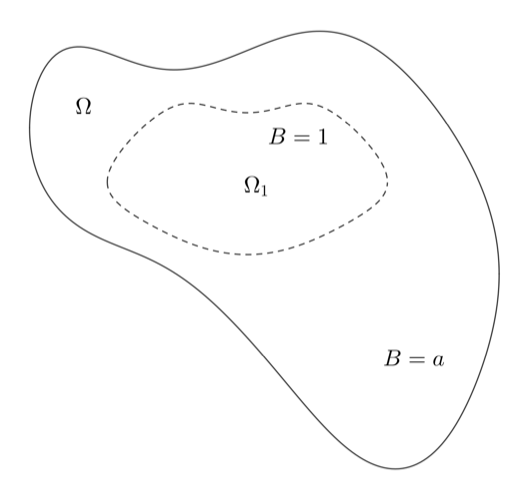}
\caption{The domain $\Omega$ is split into two parts with the edge $\Gamma$ (dashed) is a closed curve.}\label{fig3.1}
\end{figure*}
\begin{figure*}[t]
\includegraphics[width=8cm]{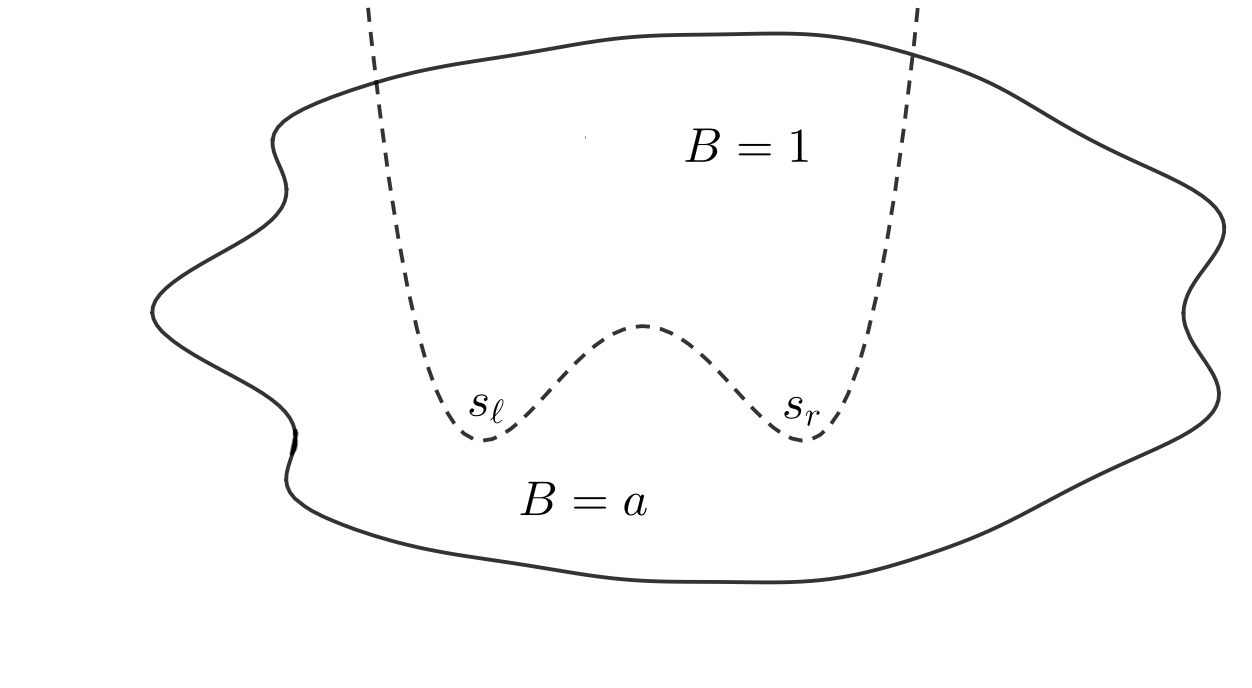}
\caption{The edge  $\Gamma$ (dashed) splits the domain $\Omega$ into two parts and intersects the boundary $\partial\Omega$ transversely.}\label{fig3.2}
\end{figure*}
We can also modify the configuration of our domains  $\Omega_1$ and $\Omega_2$ in \eqref{eq:Gam} and still get the tunneling effect but without oscillatory terms. Let $\Omega$ be a domain with a $C^1$ boundary such that $\overline{\Omega}=\overline{\Omega}_1\cup\overline{\Omega}_2$, where  $\Omega_1$ and $\Omega_2$ are disjoint simply  connected  open sets. We consider a magnetic field as in \eqref{eq:B-ms} and notice that   the edge $\Gamma=\Omega\cap\partial\Omega_1=\Omega\cap\partial\Omega_2$ (see Fig~\ref{fig3.2}). We assume that $\Gamma$  is a smooth curve and consider the Dirichlet\footnote{The Neumann realization leads to a completely different behavior, reminiscent of domains with corners \cite{A20}.} realization  of $\mathcal P_h$ in $L^2(\Omega)$. This is the situation considered in \cite{AHK}.
Now we assume that the curvature $k$  along $\Gamma$ has a  non-degenerate maximum attained at two points, with arc-length coordinates $s_\ell<0$ and $s_r=-s_\ell$, and that it is an even function in a neighborhood of $[s_\ell,s_r]$. In this situation, the splitting  between the first eigenvalues is given as follows:
\[\lambda_2(h)-\lambda_1(h)=2|w_a(h)|+o(h^{\frac{13}{8}}e^{-\mathsf{S}^a/h^{\frac 14}}),\]
where
\[w_a(h)= 2\mu_a''(\zeta_a) h^{\frac{13}{8}} \pi^{-\frac 12} g^{\frac12}_a
\mathsf{A}_a \sqrt{V_a(0)}\,e^{- \mathsf{S}_a/h^{1/4}},
 \]
 and
 \[\mathsf A_a=2\exp\left(-\int_{[ s_\ell,s_r]} \frac{ (V_a^\frac 12 )' (s)-g_a}{ \sqrt{V_a(s)}} ds\right),  \quad\mathsf S_a=\int_{[s_\ell,s_r]}\sqrt{V(s)}ds. \]

\section{Reduction to a neighborhood of the  edge}\label{sec:curv}
It will be convenient to work in Fr\'enet coordinates, $(s,t)$, along the edge $\Gamma$, valid in a neighborhood of $\Gamma$ of the form
\begin{equation}\label{eq:gam-ep}
\Gamma(\epsilon)=\{x\in\mathbb R^2~:~\textrm{ dist}(x,\Gamma)<\epsilon\}\quad(\epsilon>0)\,.
\end{equation}
Let us briefly recall these coordinates. Consider an arc-length parameterization  of $\Gamma$ , $M:(-L,L]\to\Gamma$, so that (see Assumption~\ref{ass:sym})
\[M(s_\ell)=a_{1},\quad M(s_r)=a_2,\quad  0<s_\ell<L,\quad -L<s_r<0\,, \]
and
\[\Gamma\cap\{(x,y)\in\mathbb R^2\,|\,~x=0\}=\{M(0)=:(0,y_0),M(L)=:(0,y_L)\} \quad\textrm{ with~} y_0>y_L\,. \]
Let $\mathbf n(s)$ be the unit normal to $\Gamma$ pointing inward to $\Omega_1$ (see { Fig.}~\ref{fig2}), $\mathbf t(s)=\dot{\mathbf n}(s)$ the  unit oriented tangent, so that $\textrm{ det}(\mathbf t(s),\mathbf n(s))=1$.  Let us represent the torus $\mathbb R/ 2L \mathbb Z$ by the interval $(-L,L]$.  We can pick $\epsilon_0>0$ such that 
\[\Phi: \mathbb R/ (2L \mathbb Z)\times (-\epsilon_0,\epsilon_0)\ni (s,t)\mapsto M(s)+t\mathbf n(s)\in\Gamma(\epsilon_0) \]
is a diffeomorphism whose Jacobian  is
\[\mathfrak a(s,t)=1-tk(s)\,, \]
with $k(s)$  the curvature at $M(s)$, defined by $\ddot{\mathbf n}(s)=k(s)\mathbf n(s)$.
The Hilbert space $L^2(\Gamma(\epsilon_0))$ is transformed to the weighted space  
\[L^2\big(\mathbb R/ 2L \mathbb Z\times(-\epsilon_0,\epsilon_0);\mathfrak a\,dsdt\big)\] and the  operator $\mathcal P_h$ is transformed into  the  following operator (after a gauge transformation $(u,\mathbf A)\to (v=e^{i\phi/h}u,\mathbf A'=\mathbf A-\nabla \phi)$ to eliminate the  normal component of $\mathbf A$, see \cite[App.~F]{FH10}):
\begin{multline*}
 \tilde{\mathcal P}_h:=-h^2\mathfrak a^{-1}\partial_t \mathfrak a\partial_t\\+\mathfrak a^{-1}\left(-ih\partial_s+\gamma_0-b_a(t)t+\frac{k}{2}b_a(t)t^2\right)\mathfrak a^{-1}\left(-ih\partial_s+\gamma_0-b_a(t)t+\frac{k}{2}b_a(t)t^2\right)
 \end{multline*}
where  $b_a(t)$ is introduced in \eqref{eq:potential} and $\gamma_0$ is the circulation introduced in \eqref{eq:circ}.

Following the presentation  of  \cite{BHR21}  (see also references therein),  it is convenient to introduce a truncated version of the  operator $\tilde P_h$ so that it can be defined on $\mathbb R/ 2L \mathbb Z\times\mathbb R$ instead of $\mathbb R/ 2L \mathbb Z\times(-\epsilon_0,\epsilon_0)$. This will be useful when rescaling the $t$ variable. What is handy in this situation is that the actual bound states of the operator $\mathcal P_h$ decay exponentially away from the edge, at  the length scale $\hbar:=h^{1/2}$, see \eqref{eq:dec-norm}. This  motivates the change of variables, $t=\hbar \tau$ and $s=\sigma$, that will allow the same spectral reduction as in \cite[Prop.~2.7]{BHR21}. We will skip the details which are  the same as in  \cite{BHR21}.

From now on we set
\begin{equation}\label{eq:mu(h)}
 \mu=h^{\frac14+\eta} \text{ for a fixed } \eta\in(0,\frac14)
\end{equation}
and we introduce the function
\begin{equation}\label{eq:c-delta}
c_\mu(\tau)= c(\mu\tau)\,,
\end{equation}
where $c\in C_c^\infty(\mathbb R)$ satisfies $c=1$ on $[-1,1]$  and $c=0$ on $\mathbb R\setminus(-2,2)$.
Consider the  new weight term
\[\tilde{\mathfrak a}_h(\sigma,\tau)=1-h^{1/2} c_\mu(\tau)\tau k(\sigma)\,,\]
and the self-adjoint  operator $\tilde{\mathcal N}_{h}$  on the Hilbert space $ L^2(\mathbb R/ 2L \mathbb Z\times\mathbb R;\tilde{\mathfrak a}_h  d\sigma d\tau)$,
\begin{align}\label{eq:tilde-Nh}
\tilde{\mathcal N}_{h}&=-\tilde{\mathfrak a}^{-1}_{h}\partial_\tau \mathfrak a_{\hbar}\partial_\tau \nonumber \\
&\quad+\tilde{\mathfrak a}^{-1}_{h}\left(-ih^{1/2}\partial_\sigma+{ h^{-1/2}\gamma_0}-b_a\tau +h^{1/2} c_\mu \frac{k}{2}b_a\tau^2\right)\nonumber \\
&\quad\quad \times  \tilde{\mathfrak a}^{-1}_{h}\left(-ih^{1/2}\partial_\sigma+{ h^{-1/2}\gamma_0}-b_a\tau+h^{1/2} c_\mu\frac{k}{2}b_a\tau^2\right),
\end{align}
with domain
\begin{align*}
{\textsf{ Dom}(\tilde{\mathcal N}_h)}=\{ u\in L^2(\mathbb R/ 2L \mathbb Z\times\mathbb R)& ~\vert~ \partial_\tau^2u\in L^2({ \mathbb R/ 2L \mathbb Z\times\mathbb R)}, \nonumber \\
& (-ih^{1/2}\partial_\sigma+{ h^{-1/2}\gamma_0}-b_a\tau )^2 u \in  L^2(\mathbb R/ 2L \mathbb Z\times\mathbb R)   \}.
\end{align*}
We have now  the following spectral reduction\footnote{The eigenvlaues of the operator $\tilde{\mathcal N}_h$ depend on $\eta$ in \eqref{eq:mu(h)}. However,  the estimates in Proposition~\ref{prop:BHR2.7} hold uniformly with respect  to $\eta\in(0,\epsilon)$ for any fixed $\epsilon\in(0,\frac14)$. }:
\begin{proposition}\label{prop:BHR2.7}
Let $a\in(-1,0)$ and $\mathsf S^a$ be  the Agmon distance introduced in  \eqref{defSa}. There exist $K>\mathsf{S}^a$, $C,h_0>0$ such that, for all $h\in(0,h_0)$, we have 
\[\lambda_n(h)-Ce^{-K/h^{\frac 14}}\,\leq  h\lambda_n(\tilde{\mathcal N}_h)\leq \lambda_n(h)+Ce^{-K/h^{\frac 14}}\,,\]
where    $\lambda_n(h)$ and $\lambda_n(\tilde{\mathcal N}_h)$ are the $n$-th (min-max) eigenvalues of the operators $\mathcal P_h$ and $\tilde{\mathcal N}_{h}$ respectively.
\end{proposition}
Looking at the operator in \eqref{eq:tilde-Nh}, the  effective semi-classical parameter is $\hbar=h^{1/2}$ (this is the parameter appearing in front of $\partial_\sigma$). So with
\begin{equation}\label{eq:delta}
\hbar=h^{\frac12},\quad \mu=\hbar^{\frac12+2\eta} \text{ for a fixed } \eta\in(0,\frac14)\,,
\end{equation}
we introduce the new weight term
\[\mathfrak a_\hbar(\sigma,\tau)=1-\hbar c_\mu(\tau)\tau k(\sigma)\,,\]
and the self-adjoint  operator $\mathcal N_{\hbar}$  on the Hilbert space $ L^2(\mathbb R/ 2L \mathbb Z\times\mathbb R;\mathfrak a_\hbar  d\sigma d\tau)$, which is nothing but the operator in \eqref{eq:tilde-Nh}  but with a change of parameter according to \eqref{eq:delta},
\begin{align}\label{eq:Nh}
\mathcal N_{\hbar}&=-\mathfrak a^{-1}_{\hbar}\partial_\tau \mathfrak a_{\hbar}\partial_\tau\\
&\quad+\mathfrak a^{-1}_{\hbar} \left(-i\hbar\partial_\sigma+\hbar^{-1}\gamma_0-b_a\tau +\hbar c_\mu \frac{k}{2}b_a\tau^2\right)\mathfrak a^{-1}_{\hbar}\nonumber \\
&\qquad \times\left(-i\hbar\partial_\sigma+\hbar^{-1}\gamma_0-b_a\tau+\hbar c_\mu\frac{k}{2}b_a\tau^2\right)\,.\nonumber
\end{align}
The domain of the operator $\mathcal N_{\hbar}$ is 
\begin{align*}
\textsf{ Dom}(\mathcal N_{\hbar})=\{ u\in L^2(\Gamma\times\mathbb R)~|~&\partial_\tau^2u\in  L^2(\mathbb R/ 2L \mathbb Z\times\mathbb R),  \\
&(-i\hbar\partial_\sigma+\hbar^{-1}\gamma_0-b_a\tau )^2 u \in   L^2(\mathbb R/ 2L \mathbb Z\times\mathbb R)   \}.
\end{align*}
With Proposition~\ref{prop:BHR2.7} in hand, it is enough to compute the leading order term of 
$\nu_{2}(\hbar)-\nu_{1}(\hbar)$ to prove Theorem~\ref{thm:FHK}, where, for  $n\geq 1$, we denote by  $\nu_n(\hbar)$  the $n$'th min-max eigenvalue of $\mathcal N_\hbar$. 

\section{Single well and WKB construction}\label{sec:WKB}

We will adjust the edge $\Gamma$ so that we only have a single point of maximum curvature, $s_r$ or $s_\ell$. This procedure will give us two new operators, the ``right well'' and  ``left well'' operators, $\mathcal N_{\hbar,r,\gamma_0}$ and $\mathcal N_{\hbar,\ell,\gamma_0}$ respectively. The same procedure appears,  for similar problems in the context of geometrically induced  tunneling effects \cite{HKR, KR17}, but we follow here \cite[Sec.~2.4]{BHR21} which is slightly different, but more convenient for dealing with the symbol of the operator later on.    

\subsection{Right well operator}

We present the construction for the right well operator,  $\mathcal N_{\hbar,r,\gamma_0}$ and deduce the  other one by symmetry.  Let us fix $\hat\eta$ as follows
\begin{equation}\label{eq:eta.new}
0<\hat\eta<\min\Big(\frac14,\frac{L}4\Big)\quad\textrm{ where~}L=\frac{|\Gamma|}2.
\end{equation} 
First, we identify $\Gamma$ with $(s_\ell-2L,s_\ell]$ (by periodicity and translation of the $s$ variable),  then we extend the curvature $k$ to a function $k_r$ on $\mathbb R$ as follows:
\begin{align}\label{eq:k-r}
k_r&=k\qquad \text{on}\qquad I_{2\hat\eta,r}:=(s_\ell-2L+\hat\eta,s_\ell-\hat\eta),\nonumber \\
 k_r&=0\qquad \text{on}\qquad  (-\infty,s_\ell-2L]\cup {[s_\ell,+\infty)}\,,
\end{align}
and $k_r$ has a unique non-degenerate maximum at $s_r$. Consequently, $k_r$ satisfies \eqref{eq:hyp-gd}.   

We consider now the operator in $L^2(\mathbb R^2;\mathfrak a_{\hbar,r}d\sigma d\tau)$, 
\begin{align}\label{eq:Nh-r}
\mathcal N_{\hbar,r,\gamma_0}&=-\mathfrak a^{-1}_{\hbar,r}\partial_\tau \mathfrak a_{\hbar,r}\partial_\tau\nonumber \\
&\quad +\mathfrak a^{-1}_{\hbar,r}\left(-i\hbar\partial_\sigma+\hbar^{-1}\gamma_0-b_a\tau +\hbar c_\mu \frac{k_r}{2}b_a\tau^2\right)\mathfrak a^{-1}_{\hbar,r}\nonumber \\
&\qquad \times \left(-i\hbar\partial_\sigma+\hbar^{-1}\gamma_0-b_a\tau+\hbar c_\mu\frac{k_r}{2}b_a\tau^2\right)
\end{align}
where
\begin{equation}\label{eq:ah-r}
\mathfrak a_{\hbar,r}(\sigma,\tau)=1-\hbar c_\mu(\tau)\tau k_r(\sigma)\,.
\end{equation}
Since, $k_r$ satisfies \eqref{eq:hyp-gd},
 we have, for an arbitrarily fixed $n\in\mathbb N$ (with $\beta_a$,  $c_2(a)$  and $M_3(a)$ from \eqref{eq:beta},  \eqref{eq:c2} and  \eqref{eq:m3}),
\begin{equation}\label{eq:r-well}
\lambda_n(\mathcal N_{\hbar,r,\gamma_0})=\beta_a h + k_{max} M_3(a)h^\frac 32+(2n-1)\sqrt{\frac{{ k_2}M_3(a)c_2(a)}{2}}h^\frac 74+\mathcal O(h^{\frac{15}8})\,.
\end{equation}
We are now in a simply connected domain, so the operators $\mathcal N_{\hbar,r,\gamma_0}$ and $\mathcal N_{\hbar,r,0}$ are unitarily equivalent {(we can  gauge away the flux term $\mathcal N_{\hbar,r,\gamma_0}$)}.    Denote by $u_{\hbar,r}$ a normalized ground state of  $\mathcal N_{\hbar,r,0}$ (the operator without flux term),  a corresponding normalized ground state of  $\mathcal N_{\hbar,r,\gamma_0}$ is given by:
\begin{equation}\label{eq:gs-r}
\widecheck\phi_{\hbar, r}(\sigma,\tau) =
e^{-i\gamma_0 \sigma/\hbar^2}u_{\hbar,r}(\sigma,\tau).
\end{equation}

\subsection{Left well operator}\label{sec:LWop}

Using the symmetry operator
\[U f(\sigma,\tau):=\overline{f(-\sigma,\tau)}\,, \]
we can define the left well operator on  $L^2(\mathbb R^2;\mathfrak a_{\hbar,\ell}(\sigma,\tau)$ by :
\begin{equation}\label{eq:Nh-l}
\mathcal N_{\hbar, \ell, \gamma_0}=U^{-1}\mathcal N_{\hbar, r, \gamma_0}U\,,
\end{equation}
where
\[ \mathfrak a_{\hbar,\ell}(\sigma,\tau)=\mathfrak a_{\hbar,r}(-\sigma,\tau)\,.\]
The left and right operators have the same spectrum, and a normalized ground state of $\mathcal N_{\hbar, \ell, \gamma_0}$ is 
\begin{equation}\label{eq.phil0}
\widecheck\phi_{\hbar,\ell}:=U\widecheck\phi_{\hbar,r}=e^{-i\gamma_0 \sigma/\hbar^2}u_{\hbar,\ell}(\sigma,\tau)
\end{equation}
 where $u_{\hbar,\ell}=Uu_{\hbar,r}$.
 
\subsection{WKB expansions}

We focus on the right well operator and construct an approximate eigenvalue and an approximate ground state by WKB expansions, involving \emph{formal series} in  the sense of \cite[Notation~1.13]{BHR15}. The construction can be translated to the left operator by symmetry.  

Let us introduce the Agmon distance 
\begin{equation}\label{eq:Ag-r}
\Phi_r(\sigma)=\int_{[s_r,\sigma]}\sqrt{V_{a,r}(s)}ds
\end{equation} 
related to the  ``right well'' potential\footnote{Recall from \eqref{eq:m3} that $M_3(a)<0$, so $V_{a,r}\geq 0$.}
\begin{equation}\label{eq:pot-r}
V_{a,r}(\sigma)=\frac{2M_3(a)(k_r(s)-k_{\max})}{\mu''_a(\zeta_a)}
\end{equation} 
\begin{theorem}\label{thm:WKB}
There exist two sequences $ (b_j)_{j\geq 0}\subset \textsf{ Dom}(\mathcal N_{\hbar, r})$, $(\delta_j)_{j\geq0}\subset\mathbb R$, 
a  family of functions $(\Psi_{\hbar,r})_{\hbar\in(0,\hbar_0]}\subset L^2(\mathbb R^2)$ and a family of  real numbers $(\delta(\hbar))_{\hbar\in(0,\hbar_0]}$ such that
\begin{equation}\label{eq.psir}
e^{\Phi_r(\sigma)/\hbar^{\frac 12}} e^{-i\sigma\zeta_a/\hbar}\Psi_{\hbar,r}(\sigma,\tau)\underset{\hbar\to 0}{\sim}\hbar^{-\frac18} \sum_{j\geq 0} b_{j}(\sigma,\tau) \hbar^{\frac j2},
\end{equation}
\[ \delta(\hbar)\underset{\hbar\to 0}{\sim}\sum_{j\geq 0}\delta_{j}\hbar^{\frac j2}\,,\]
and
\begin{equation}\label{eq:Nh-WKB}
e^{\Phi_r(\sigma)/\hbar^{\frac 12}}\left(\mathcal{N}_{\hbar, r}-\delta(\hbar)\right) \Psi_{\hbar,r}=\mathcal{O}(\hbar^\infty)\,.
\end{equation}
Furthermore
\[\delta_{0}=\beta_a\,,\quad \delta_{1}=0,\quad \delta_2=M_3(a) k_{\max},\quad
\delta_3=\sqrt{\frac{{ k_2}M_3(a)c_2(a)}{2}}, \]
\begin{equation}\label{eq.an0}
b_{0}(\sigma,\tau)=f_{0}(\sigma)\phi_a(\tau),
\end{equation}
and $f_{0}$ solves the effective transport equation 
\begin{equation}\label{eq.effectiveT}
\frac{ \mu''_a(\zeta_a)}{2}(\Phi'_r\partial_\sigma+\partial_{\sigma}\Phi'_r)f_{0}+ iF(\sigma)f_{0}= \sqrt{\frac{{ k_2}M_3(a)c_2(a)}{2}}f_{0}\,,
\end{equation}
where $F$ is a smooth real-valued function, introduced in \eqref{eq:def-F},  such that $F(s_r)=0$.
\end{theorem}
\begin{remark}\label{rem:WKB-not}
Let us explain precisely how   the asymptotics in Theorem~\ref{thm:WKB} are interpreted. For  every  $N\geq 1$ we introduce  the function $\psi_{\hbar,r}^N(\sigma,\tau)$ and the real number $\delta^N(\hbar)$ as follows:
\[  
\Psi^N_{\hbar,r}(\sigma,\tau):= e^{-\Phi_r(\sigma)/\hbar^{\frac 12}} e^{i\sigma\zeta_a/\hbar}\, \hbar^{-\frac18} \sum_{j=0}^Nb_{j}(\sigma,\tau)\hbar^{\frac j2},\quad  \delta^N(\hbar)=\sum_{j=0}^N\delta_{j}\hbar^{\frac j2}\,.  \]
 Then \eqref{eq:Nh-WKB} means
 \[ e^{\Phi_r(\sigma)/\hbar^{\frac 12}}\left\|\left(\mathcal{N}_{\hbar, r}-\delta(\hbar)\right) \Psi_{\hbar,r}^N\right\|_{L^2(\mathbb R_\tau)}=\mathcal{O}(\hbar^N)\]
 locally uniformly with respect to $\sigma$.
\end{remark}
\begin{proof}[Proof of Theorem~\ref{thm:WKB}]
We work in an arbitrary bounded set of $\mathbb R^2$,
so, in the below computations,  we take $c_\mu=1$  in  \eqref{eq:Nh-r}  at  the cost of an error  $\mathcal O(h^\infty)$.   That is possible because our constructions will involve   functions  decaying  exponentially with respect  to the normal  variable $\tau$.
 
Let us introduce the operator 
\begin{equation*}
\widehat{\mathcal{N}}_{\hbar, r}:=e^{\Phi_r(\sigma)/\hbar^{\frac 12}} e^{i\sigma\zeta_a/\hbar} \mathcal{N}_{\hbar, r}e^{-i\sigma\zeta_a/\hbar}e^{-\Phi_r(\sigma)/\hbar^{\frac 12}} \,.
\end{equation*}
It admits the formal expansion
\begin{equation*}
\widehat{\mathcal{N}}_{\hbar, r}= \mathcal L_0+\hbar^{1/2}\mathcal  L_1+\hbar\mathcal L_2+\hbar^{3/2}\mathcal L_3+\hbar^2\mathcal L_4+\cdots
\end{equation*}
where
\begin{align*}
\mathcal L_0&=-\partial_\tau^2+(\zeta_a+b_a\tau)^2\\
\mathcal L_1&=-2(\zeta_a+b_a\tau)i\Phi'_r(\sigma)\\
\mathcal L_2&=k_r\partial_\tau-2(\zeta_a+b_a\tau)\Big( -i\partial_\sigma+\frac{k_r}{2}b_a\tau^2\Big)
-\Phi'_r(\sigma)^2+2k_r\tau(\zeta_a+b_a\tau)^2\\
\mathcal L_3&=\Big(-i\partial_\sigma+\frac{k_r}2b_a\tau^2\Big)i\Phi'(\sigma)+i\Phi'_r(\sigma)\Big(-i\partial_\sigma+\frac{k_r}2b_a\tau^2\Big)\\
&\qquad-4\Phi'_r(\sigma)\tau k_r(\zeta_a+b_a\tau)\\
\mathcal L_4&=-\partial_\sigma^2+2k_r^2\tau^2(\zeta_a+b_a\tau)^2 \\
&\qquad-(\zeta_a+b_a\tau)\Big[ \Big(-i\partial_\sigma+\frac{k_r}2b_a\tau^2\Big)k_r+k_r\Big(-i\partial_\sigma+\frac{k_r}2b_a\tau^2\Big) \Big]\\
&\,\,\vdots
\end{align*}
Let 
$b(\sigma,\tau;\hbar):=\sum_{j\geq 0}b_{j}(\sigma,\tau) \hbar^{\frac j2}$ and let us formally solve the equation
\[ \big(\widehat{\mathcal{N}}_{\hbar, r}-\delta(\hbar)\big)b(\sigma,\tau;\hbar)=\mathcal O(\hbar^\infty).\]
Expanding  the foregoing equation in powers of  $\hbar^{1/2}$, the vanishing of the coefficient of each $\hbar^{j/2}$, $j\geq 0$, yields the following equations

\begin{align*}
(\mathcal L_0-\delta_0)b_0&=0\\
(\mathcal L_0-\delta_0)b_1&=(\delta_1-\mathcal L_1)b_0\\
(\mathcal L_0-\delta_0)b_2&=(\delta_2-\mathcal L_2)b_0+(\delta_1-\mathcal L_1)b_1\\
(\mathcal L_0-\delta_0)b_3&=(\delta_3-\mathcal L_3)b_0+(\delta_2-\mathcal L_2)b_1+(\delta_1-\mathcal L_1)b_2\\
&\,\,\vdots
\end{align*}
We will find solutions to these equations one by one.
The first equation leads us  to choose $\delta_0=\zeta_a$ and $b_0(\sigma,\tau)=f_0(\sigma)\phi_a(\tau)$, where $f_0(\sigma)$ is to be determined at a later stage. The function $f_0$ will actually be free untill the first equation involving $\mathcal L_3$.

For the equation for $b_1$, we determine $\delta_1$ by assuming that $(\delta_1-\mathcal L_1)b_0$ is orthogonal  to $\phi_a$ in $L^2(\mathbb R)$. Then, we take the inner product with $\phi_a$ in  $L^2(\mathbb R)$, use \eqref{eq:m0} and  get
\[ \delta_1=0,\quad b_1(\sigma,\tau)=2i\Phi'_r(\sigma)f_0(\sigma) {\mathcal R_a}\big((\zeta_a+b_a\tau)\phi_a\big),\]
where $\mathcal R_a$ the regularized resolvent introduced in \eqref{eq:R}.   Since we are applying $\mathcal R_a$ on functions orthogonal to $\phi_a$,  we can slightly abuse notation and say that it is equal to  $(\mathcal L_0-\delta_0)^{-1}$.

The equation for $b_2$ will  determine $\delta_2$. This equation can be solved if $(\delta_2-\mathcal L_2)b_0-\mathcal L_1b_1$ is orthogonal to $\phi_a$ in $L^2(\mathbb R)$, which we assume henceforth.  Taking the inner product with $\phi_a$ in  $L^2(\mathbb R)$,  using \eqref{eq:I2} and Remark~\ref{rem:AHK2.3}, we get
\[
\delta_2f_0(\sigma)+\frac{\mu_a''(\zeta_a)}{2}\Phi'_r(\sigma)^2f_0(\sigma)-M_3(a)k_r(\sigma)f_0(\sigma)=0,
\]
since
\[ 2\int_{\mathbb R}\tau(\zeta_a+b_a(\tau) \tau)^2|\phi_a(\tau)|^2\,d\tau-\int_{\mathbb R}b_a(\tau) \tau^2(\zeta_a+b_a(\tau) \tau)|\phi_a(\tau)|^2\,d\tau=M_3(a)\,.\]
So, we choose $\delta_2=k_r(0)=k_{\max}$, and the foregoing  equation  involving  $f_0$  is valid  everywhere in light of \eqref{eq:Ag-r}, independently of the choice of $f_0$. At the same time, we choose $b_2$ as follows:
\[b_2(\sigma,\tau)=(\mathcal L_0-\delta_0)^{-1}\Big((\delta_2-\mathcal L_2)b_0-\mathcal L_1b_1 \Big). \]
From the equation of $b_3$, we will determine $\delta_3$ and $f_0(\sigma)$. Taking the inner product with $\phi_a$ in $L^2(\mathbb  R)$  and using \eqref{eq:I2}, we  get
\begin{multline*}
\Big\langle  (\delta_3-\mathcal L_3)b_0+(\delta_2-\mathcal L_2)b_1+(\delta_1-\mathcal L_1)b_2, \phi_a\Big\rangle_{L^2(\mathbb R)}=\\
\Big(\delta_3-\frac{\mu''(\zeta_a)}2\big(\Phi'_r\partial_\sigma+\partial_\sigma\Phi'_r \big)\Big)f_0(\sigma)+iF(\sigma)f_0(\sigma),
\end{multline*}
where  $F(\sigma)$ is the real-valued function
\begin{equation}\label{eq:def-F}
F(\sigma)=|\Phi'_r(\sigma)|^2\int_{\mathbb R}  g(\sigma,\tau) \phi_a(\tau)d\tau,
\end{equation}
and
\begin{align*}
 g(\sigma,\tau)&=(\mathcal L_0-\delta_0)^{-1}\big(g_1(\sigma,\tau)+g_2(\sigma,\tau)\big)\\
g_1(\sigma,\tau)&=
-\Big(k_{\max}+k_r(\sigma)\big(\zeta_a+b_a\tau)^2-b_a(\zeta_a+b_a\tau)\tau^2\big)-|\Phi(\sigma)|^2\Big)\phi_a(\tau)\\
&\qquad+k_r(\sigma)\phi_a'(\tau)\\
g_2(\sigma,\tau)&=-4|\Phi_r(\sigma)|^2(\mathcal L_0-\delta_0)^{-1}\big((\zeta_a+b_a\tau)\phi_a(\tau)\big).
\end{align*}
Since $\Phi_r'(s_r)=0$, we observe that $F(s_r)=0$. We can solve the equation of $b_3$ if  $(\delta_3-\mathcal L_3)b_0+(\delta_2-\mathcal L_2)b_1+(\delta_1-\mathcal L_1)b_2$ is orthogonal to $\phi_a$ in $L^2(\mathbb R)$,  which yields the following equation for $f_0$:
\[\Big(\delta_3-\frac{ \mu_a''(\zeta_a)}2\big(\Phi'_r\partial_\sigma+\partial_\sigma\Phi'_r \big)\Big)f_0(\sigma)+F(\sigma)f_0(\sigma)=0.\]
Since $F(s_r)=\Phi'_r(s_r)=0$, the foregoing  equation has a solution satisfying $f_0(s_r)\not=0$ if $\delta_3= \frac{\mu_a''(\zeta_a)}{2}\Phi''_r(s_r)$, thereby determining $\delta_3$ (from \eqref{eq:pot-r}) and $f_0$.
The procedure can be continued to any preassigned order.
\end{proof}

\begin{remark}[Solving \eqref{eq.effectiveT} \& normalization of $\Psi_{\hbar,r}$]\label{rem:phase}~

In \eqref{eq.effectiveT}, we make the ansatz $f_0(\sigma)=e^{i\alpha_0(\sigma)}\tilde f_0(\sigma)$, with $\tilde f_0$ and $\alpha_0$ are real-valued functions such that $\tilde f_0(0)>0$.  Then we get from \eqref{eq.effectiveT}:
\begin{equation*}
\frac{\mu''_a(\zeta_a)}2\big( \Phi'_r\partial_\sigma+\partial_{\sigma}\Phi'_r\big)\tilde f_0(\sigma)= \sqrt{\frac{{ k_2}M_3(a)c_2(a)}{2}}f_{0}(\sigma)
\end{equation*}
and 
\begin{equation*}
\mu''_a(\zeta_a)\alpha_0'(\sigma)\Phi_r'(\sigma)+F(\sigma)=0\,.
\end{equation*}

This will determine $\tilde f_0(\sigma)$ uniquely up to  the choice of  $\tilde f_0(0)$, and also $\alpha_0(\sigma)$ uniquely up to an additive constant (see  \cite[Eq.~(2.14)~\&~Rem.~2.9]{BHR21}).
We choose $\tilde f(0)=\left(\frac{g_a}{\pi}\right)^{1/4}\left(\mathsf A_\textsf{ u}^a\right)^{1/2}$ which yields that the WKB solution  $\Psi_{\hbar,r}$ in Theorem~\ref{thm:WKB} is almost normalized, $\|\Psi_{\hbar,r}\|\sim 1$. The constant $\alpha_a$ appearing in Theorem~\ref{thm:FHK} is
\begin{equation}\label{eq:alpha}
\alpha_a=\frac{\alpha_{0}(0)-\alpha_0(-L)}{L}.
\end{equation}
\end{remark}

\section{Optimal tangential Agmon estimates}\label{sec:dec}

The   challenge   of obtaining optimal decay estimates of bound states of the Neumann magnetic Laplacian matching with the WKB solutions was recently  taken up in  \cite{BHR21}  by  introducing pseudo-differential calculus  with operator-valued symbols. Fortunately, the method is quite general and can handle our situation  of magnetic steps.

\subsection{A tangential elliptic estimate}

We work with the single `right well' flux free operator, $\mathcal N_{\hbar,r}:=\mathcal  N_{\hbar,r,0}$, introduced in \eqref{eq:Nh-r}. For the sake of simplicity, we will omit the reference to `right well' in the notation and write $\mathcal N_{\hbar}$ and  $k$ instead of $\mathcal N_{\hbar,r}$ and $k_r$.

 The optimal estimates, on the bound states  of  $\mathcal N_{\hbar}$, will hold in spaces with an exponential weight, defined via  a sub-solution of an effective eikonal  equation. More precisely, we  consider a family of Lipschitz functions $(\varphi_\hbar)_{h\in(0,1]}\subset C(\mathbb R;\overline{\mathbb R_+})$ satisfying the following hypothesis:
\begin{assumption}\label{ass:phi}  For all $M>0$ there exist $\hbar_0,C,R>0$ such that, for all $\hbar\in(0,\hbar_0)$, the  function $\varphi:=\varphi_\hbar$ satisfies
	\begin{enumerate}[\rmfamily (i)]
		\item for all $\sigma\in\mathbb R$, $\mathfrak{v}(\sigma)-\frac{\mu_a''(\zeta_a)}{2}\varphi'(\sigma)^2\geq 0$, 
		\item for all $\sigma$ such that $|\sigma-s_r|\geq R\hbar^{\frac 12}$, $\mathfrak{v}(\sigma)-\frac{\mu_1''(\xi_0)}{2}\varphi'(\sigma)^2\geq M \hbar$,
		\end{enumerate}	
where $\mathfrak{v}(\sigma)=M_3(a)(k(\sigma)-k_{\max})$.
\end{assumption}
In the sequel, to  lighten the notation, we write   $\varphi$ instead of $\varphi_\hbar$. We consider the  conjugate operator, with the same domain as  $\mathcal  N_{\hbar}$, and defined by:
\begin{equation}\label{eq:conj-op}
\mathcal N_{\hbar}^\varphi=e^{\varphi/\hbar^{\frac12}}\mathcal N_{\hbar}e^{-\varphi/\hbar^{\frac12}}
=-a_\hbar^{-1}\partial_\tau a_\hbar\partial_\tau+ a_\hbar^{-1}\mathcal T_{\hbar}^\varphi a_\hbar^{-1}\mathcal T_{\hbar}^\varphi
\end{equation}
where
\[ \mathcal T_{\hbar}^\varphi:=
\left(-i\hbar\partial_\sigma-b_a\tau+i\hbar^{\frac 12}\varphi'+\hbar c_\mu\frac{\kappa_r}{2}b_a\tau^2\right)
\]
\begin{theorem}[Bonnaillie-No\"{e}l--H\'erau--Raymond]\label{thm:dec}
Let $c_0>0$ and $\chi_0\in C_c^\infty(\mathbb  R)$ be $1$ in a neighborhood of $0$. Under Assumption \ref{ass:phi}, there exist $c, \hbar_0>0$ such that for all $\hbar\in(0,\hbar_0)$, $z\in[ \beta_a+M_3(a)k_{\max}\hbar-c_0\hbar^2,\beta_a+M_3(a)k_{\max}\hbar+c_0\hbar^2]$ and all $\psi\in\textsf{ Dom}(\mathcal{N}_\hbar^{\varphi})$,
\[
c\hbar^2\|\psi\|\leq \|\langle\tau\rangle^6(\mathcal{N}^{\varphi}_\hbar-z)\psi\|
+\hbar^2\|\chi_0(\hbar^{-\frac 12}R^{-1}(\sigma-s_r)) \psi\|\,,
	\]
	and
	\[
	c\hbar^{2}\|\hbar^2 \partial_\sigma^2\psi\|\leq \|\langle\tau\rangle^6(\mathcal{N}^{\varphi}_\hbar-z)\psi\|
	+\hbar^2\|\chi_0(\hbar^{-\frac 12}R^{-1}(\sigma-s_r)) \psi\|\,,
		\]
where $\langle \tau\rangle=(1+\tau^2)^{1/2}$.
\end{theorem}
Modulo the decomposition of  the symbol of the operator $\mathcal  N_{\hbar}^\varphi$ and its parametrix, the proof of Theorem~\ref{thm:dec} is the same  as that of \cite[Thm.~5.1]{BHR21}. In the sequel, we give only the new ingredients.

Let us write 
\begin{equation} 
\mathcal N_{\hbar}^\varphi u =\textrm{ Op}^\textrm{ W}_{\hbar}(n_{\hbar}) \,u =\frac{1}{(2\pi\hbar)} \iint_{\mathbb  R^2} e^{i(\sigma-s)\cdot \xi/\hbar} n_\hbar \left(\frac{\sigma+s}{2}, \xi\right)u(s)   ds d\xi
\end{equation}
where the foregoing quantization formula is formal, unless we consider it on, say,  for $u$ in the space
$
\mathcal  S\big(\mathbb R;\widehat{\mathcal S}(\mathbb R)\big)$ where
\begin{equation}\label{eq:hat-S}
\widehat{\mathcal S}(\mathbb R)=\{v\in H^2(\mathbb R))\,|~v|_{\overline{\mathbb R_\pm}}\in\mathcal S(\overline{\mathbb R_\pm})\}.
\end{equation}
The operator-valued symbol  $n_{\hbar}$   can be decomposed as follows
\[n_\hbar=n_0+\hbar^{\frac 12}n_1+\hbar n_2+  \hbar^{\frac 32}n_3 + \hbar^2 \tilde{r}_\hbar  ,\]
 where 
\begin{equation}\label{eq.nj}
\begin{aligned}
n_0(\sigma,\xi)&=-\partial_\tau^2+(\xi-b_a(\tau)\tau)^2\,,\\
n_1(\sigma,\xi)&=2i(\xi-b_a(\tau)\tau)\varphi'(\sigma)\,,\\
n_2(\sigma,\xi)&=-\varphi'(\sigma)^2+\kappa c_\mu(\tau) \partial_\tau+c_\mu \kappa(\sigma)(\xi-b_a(\tau)\tau)b_a(\tau)\tau^2\\
&\quad+2\kappa(\sigma)\tau c_\mu (\tau)(\xi-b_a\tau)^2  +\kappa(\sigma) \tau c_\mu'(\tau )\,, \\
\Re\, n_3(\sigma,\xi)& = 0\,, \\
 \tilde{r}_{\hbar}(\sigma,\xi) &= {\mathcal O}(\tau^4, (\xi-b_a(\tau)\tau)^2 \tau^2, (\xi-b_a\tau)\tau, \tau^2\partial_\tau).
\end{aligned}
\end{equation}
The notation  ${\mathcal O}$ is defined in \cite[Notation~3.1]{BHR21}:\\
{\itshape For differential operators $A,B,C,\cdots$  on $\mathbb  R_\tau$,  writing $A=\mathcal  O(B,C,\cdots)$ means the following:
\begin{equation}\label{eq:not-O-op}
\exists\,c> 0, ~\forall\,u\in \widehat{\mathcal S}(\mathbb  R),\quad\|Au\|_{L^2(\mathbb R)}\leq c\big(\|Bu\|_{L^2(\mathbb R)}+\|Cu\|_{L^2(\mathbb R)}+\cdots \big),\end{equation}
where $\widehat{\mathcal S}(\mathbb  R)$ is the space introduced in \eqref{eq:hat-S} and  the constant $c$ is independent of $A,B,C,\cdots$ (in particular, in \eqref{eq.nj},  the estimate is uniform with respect to $(\sigma,\xi)$).}

Let  us introduce  a modified symbol by  truncating the frequency variable. Recall that $\zeta_a<0$ is the unique minimum of the model operator with a flat edge (see \eqref{eq:beta}). It will be  convenient to introduce
\begin{equation}\label{eq:hat-zeta}
\hat\zeta_a:=-\zeta_a>0\,.
\end{equation}
 Pick a smooth bounded and increasing function $\chi\in  C^\infty(\mathbb R)$ such that $\chi(\xi)=\xi$ for $\xi\in (-\hat\zeta_a/2,\hat\zeta_a/2)$, and  $\eta_+:=\lim_{\xi\to+\infty}\chi(\xi)\in(0,\hat\zeta_a)$. We introduce the function 
\begin{equation}\label{eq:chi1}
\chi_1(\xi)=\hat\zeta_a+\chi(\xi-\hat\zeta_a),
\end{equation}
and the operator
\begin{equation}\label{eq:new-symbol}
\textrm{ Op}^\textrm{ W}_{\hbar}(p_{\hbar})\quad\textrm{  where~}p_{\hbar}(\sigma,\xi):=n_{\hbar}(\sigma,\chi_1(\xi)).
\end{equation}
Now consider the Grushin problem defined by the matrix operator
\begin{equation}\label{eq:full-symbol}
\mathcal P_z(\sigma,\xi)=\left(\begin{array}{ll} p_\hbar-z&\cdot v_\xi\\
\langle\cdot,v_\xi\rangle&0 \end{array}\right)\in\mathcal S\big(\mathbb R_{\sigma,\xi}^2,\mathcal L(\textsf{ Dom}(p_0)\times \mathbb C,L^2(\mathbb R)\times\mathbb C) \big)
\end{equation}
where $\mathcal S\big(\mathbb R_{\sigma,\xi}^2,\mathcal L(\textsf{ Dom}p_0\times \mathbb C,L^2(\mathbb R)\times\mathbb C) \big)$ is defined in \cite[Notation~3.1]{BHR21},
\begin{equation}\label{eq:p0}
 p_0(\sigma,\xi) =-\partial_\tau^2+(\chi_1(\xi)-b_a\tau)^2
\end{equation}
is the principal symbol of $p_\hbar$ and $v_\xi$ is the positive normalized ground state of $p_0$, with corresponding eigenvalue $\mu_1(\chi_1(\xi))=\mu_a(-\chi_1(\xi))$ (see \eqref{mu_a_1}).

From the decomposition of $n_\hbar$, we can decompose $\mathcal P_z$ as follows:
\[ \mathcal P_z=\mathcal P_z^{[3]}+\hbar^2\mathcal R_\hbar,\quad 
 \mathcal P_z^{[3]}=\mathcal P_{0,z}+\hbar^{1/2} \mathcal P_1+\hbar\mathcal P_2
 +\hbar^{3/2} \mathcal P_3,\]
 where
 \[\mathcal  P_{0,z}=\left(\begin{array}{ll} p_0-z&\cdot v_\xi\\
\langle\cdot,v_\xi\rangle&0 \end{array}\right) ,\quad \forall j\geq 1,~\mathcal P_j=\left(\begin{array}{ll} p_j&0\\
0&0 \end{array}\right),\quad \mathcal R_\hbar= \left(\begin{array}{ll} r_\hbar&0\\
0&0 \end{array}\right),\]
and
\begin{equation}\label{eq.pj}
\begin{aligned}
p_1&=2i(\chi_1(\xi)-b_a\tau)\varphi',\\
p_2&=-\varphi'^2+\kappa c_\mu \partial_\tau+c_\mu \kappa(\chi_1(\xi)-b_a\tau)b_a\tau^2+2\kappa\tau c_\mu (\chi_1(\xi)-b_a\tau)^2  \nonumber \\
&\quad +\kappa \tau c_\mu'\left(\tau \right), \\
 \textrm{ Re}\, p_3 & = 0, \\
 \tilde{r}_{\hbar} &= {\mathcal O}(\tau^4,  \tau^2\partial_\tau),
\end{aligned}
\end{equation}
where $\mathcal O (\tau^4,  \tau^2\partial_\tau)$ is understood in the sense of \eqref{eq:not-O-op}.

Then  one can construct a parametrix of $\textrm{ Op}^\textrm{ W}_{\hbar}(\mathcal P_z)$ (see \cite[Thm.~3.5]{BHR21} for details)
\[ \mathcal L_z^{[3]}=\left(\begin{array}{ll} q_z&q_z^+\\
q_z^-&q_z^\pm \end{array}\right),\quad  \textrm{ Op}^\textrm{ W}_{\hbar}(\mathcal L_z^{[3]}) \, \textrm{ Op}^\textrm{ W}_{\hbar}(\mathcal P_z)=\textrm{ Id}+\hbar^2\mathcal O(\langle\tau\rangle^6),\]
where
\[q_z^\pm=q_{0,z}^\pm+\hbar^{1/2}q_{1,z}^\pm+\hbar q_{2,z}^\pm +\hbar^{3/2}q_{3,z}^\pm\,,\]
with
\begin{equation}\label{eq.qj}
\begin{aligned}
q_{0,z}^\pm&=z-\mu_1(\chi_1(\xi)),\\
q_{1,z}^\pm&=-i\varphi'(\sigma)\partial_\xi\mu_1(\chi_1(\xi)),\\
q_{2,z}^\pm&=k(\sigma)C_1(\xi,\mu)+C_2(\xi,z)\varphi'(\sigma)^2,\\
\textrm{ Re} \,q_{3,z}^\pm&=0,
\end{aligned}
\end{equation}
and
\begin{equation}\label{eq.Cj}
\begin{aligned}
C_1(\xi,\mu)=&-\langle \big(c_\mu \partial_\tau+c_\mu (\chi_1(\xi)-b_a\tau)b_a\tau^2+2\tau c_\mu (\chi_1(\xi)-b_a\tau)^2\big)v_\xi,v_\xi\rangle\\
&\qquad-\langle\tau c_\mu'\partial_\tau\big)v_\xi,v_\xi\rangle, \\
C_2(\xi,\mu)=&1-\langle(p_0-z)^{-1}\Pi^\bot(\chi_1(\xi)-b_a\tau)v_\xi,(\chi_1(\xi)-b_a\tau)v_\xi\rangle.
\end{aligned}
\end{equation}
Here $\Pi=\Pi_\xi$  is the orthogonal  projection on  $v_\xi$  and $\Pi^\bot=\textrm{ Id}-\Pi$.  Note that, by  \eqref{eq:hat-zeta}, Remark~\ref{rem:AHK2.3} and \eqref{eq:I2},
\[C_1(\hat\zeta_a,0)=-M_3(a),\quad  C_2(\hat\zeta_a,\beta_a)=\frac{\mu_a''(\zeta_a)}{2}.\] 
Now we argue like \cite[Prop~4.4]{BHR21}. Recall that $|z-\beta_a-M_3(a)k_{\max}\hbar|\leq  c_0\hbar^2$ and that $\mu\to0$ as $h\to0$. Expanding $C_1(\xi,\mu)$ and $C_2(\xi,z)$ near $\xi=\hat\zeta_a$, we  get
\begin{align*}
C_1(\xi,\mu)&=-M_3(a)k_{\max}\hbar+\mathcal O\big(\hbar\min(1,|\xi-\hat\zeta_a|)\big),\\
C_2(\xi,z)&=\frac{\mu_a''(\zeta_a)}{2}+\mathcal O\big(\hbar\min(1,|\xi-\hat\zeta_a|)\big).
\end{align*}
Furthermore, since $\mu'_a(\zeta_a)=0$ and $\mu''_a(\zeta_a)>0$, there exists a constant $c_1>0$ such that
\[\mu(\chi_1(\xi))-z \geq c_1\min(1,|\xi-\hat\zeta_a|^2).  \]
Now we have the following lower bound
\[ -\textrm{ Re}\,q_z^\pm\geq  \hbar(\mathfrak  v(\sigma)-C_2(\hat\zeta_a,\beta_a)\varphi'(\sigma)^2\big )-C\hbar^2,\]
where $\mathfrak  v(\sigma)$  is introduced in Assumption~\ref{ass:phi}. 
We apply the Fefferman-Phong inequality  \cite[Thm.~3.2]{B} (see also \cite[Thm.~4.3.2]{Z}) on the symbol 
\[\mathfrak a(\hat\sigma,\hat\xi;\hbar):=\mathscr{A}(\hbar^{1/2}\hat\sigma,\hbar^{1/2}\hat\xi;\hbar), \]
where
\[\mathscr{A}(\sigma,\xi;\hbar):=-\textrm{ Re}\, q_z^\pm(\sigma,\xi) -\hbar(\mathfrak  v(\sigma)-C_2(\hat\zeta_a,\beta_a)\varphi'(\sigma)^2\big )-C\hbar^2  .\]
In  that way, we have 
\[-\textrm{ Re}\langle  \textrm{ Op}_{\hbar}^\textrm{  W}(q_z^\pm)\psi,\psi\rangle \geq \hbar\int_{\mathbb R}\left(\mathfrak v(\sigma)-\frac{\mu''_a(\zeta_a)}{2}|\varphi'(\sigma)|^2{-C\hbar}\right)|\psi|^2d\sigma, \]
from which  we get the following estimate (see  \cite[Thm.~4.2]{BHR21})
\[cR^2\hbar^2\|\psi\|\leq \|(\textrm{ Op}_{\hbar}^\textrm{  W} (p_\hbar)-z)\psi \|+C_R\hbar^2\|\chi_0
(\hbar^{-1/2}R^{-1}(\sigma-s_r))\psi\|+\hbar^2\|\tau^6\psi\|\,, \]
which is almost the inequality in  Theorem~\ref{thm:dec}, but with the  operator $\textrm{ Op}_{\hbar}^\textrm{  W}(p_\hbar)$ instead  of the operator  $\mathcal N_{\hbar}^\varphi=\textrm{ Op}_{\hbar}^\textrm{  W}(n_\hbar)$. The only difference between the two operators is the frequency cut-off in the symbol, which can be removed following the same argument    in \cite[Thm.~5.1]{BHR21}.

\subsection{Applications}

By appropriate choices of the function $\varphi$ in Theorem~\ref{thm:dec}, we get optimal tangential estimates  for the bound states of  the `single' and `double' well operators. For details, see \cite[Corol.~5.7, Corol.~6.1\,\&\,Prop.~6.2]{BHR21}.

\begin{proposition}[Decay of bound states]\label{prop:dec-bs}
Let $\theta,\varepsilon\in(0,1)$ and $K>0$. There exist $C,\hbar_0>0$ such that for all $\hbar\in(0,\hbar_0)$, the  following is true. 

If $\lambda$ eigenvalue of the operator $\mathscr{N}_
{\hbar}$ in \eqref{eq:Nh}, $\left|\lambda-\left(\beta_a+M_3(a)k_{\max}\hbar\right)\right|\leq K\hbar^2$ and  $u\in\textsf{ Dom}(\mathscr{N}_{\hbar})$ is an eigenfunction associated to  $\lambda$, then
\[\int_{[-L,L)\times\mathbb R}e^{2\varphi/\hbar^{\frac 12}}|u|^2dsd\tau\leq Ce^{\varepsilon/\hbar^{\frac 12}}\|u\|^2_{L^2([-L,L)\times\mathbb R)}\,.\]
where
\[\varphi=(1-\theta)^{1/2}\min(\tilde\Phi_r,\tilde\Phi_\ell)\,,\]
and $\tilde\Phi_r,\tilde\Phi_\ell$  are $2L$-periodic functions satisfying, for $\eta$  sufficiently small,
\[
\begin{aligned}
\tilde\Phi_r(\sigma)\,|_{-L\leq \sigma\leq s_\ell-\eta}&=\Phi_r(\sigma):=\sqrt{\frac{-2M_3(a)}{\mu_a''(\zeta_a)}}\int_{[s_r,\sigma]}\sqrt{k_{\max}-k(s)}\,ds,\\  
\tilde\Phi_\ell(\sigma)\,|_{-L\leq \sigma\leq s_r-\eta}&=\Phi_\ell(\sigma):=\sqrt{\frac{-2M_3(a)}{\mu_a''(\zeta_a)}}\int_{[s_\ell,\sigma]}\sqrt{k_{\max}-k(s)}\,ds\,.  
\end{aligned}
\]
\end{proposition}
\begin{remark}\label{rem:dec-bs}
The  estimate in Proposition~\ref{prop:dec-bs} continues to hold if $\lambda$ is an eigenvalue of the right or left well operator, $\mathcal N_{\hbar,r}$ or $\mathcal N_{\hbar,\ell}$, and $u$ is a corresponding eigenfunction.
\end{remark}
Returning to the `one well' operators, $\mathcal  N_{\hbar,r,\gamma_0}$ and $\mathcal N_{\hbar,\ell,\gamma_0}$ introduced in  \eqref{eq:Nh-r} and \eqref{eq:Nh-l} respectively, we get  from Proposition~\ref{prop:dec-bs} and the min-max principle, the following rough estimate, important for the analysis of tunneling later on,
\begin{equation}\label{eq:ev-sw}
\mu_1^\textrm{ sw}(\hbar)-\tilde{\mathcal O}(e^{-\mathsf{S}^a/\sqrt{\hbar} })\leq \nu_{1,a}(\hbar)\leq \nu_{2,a}(\hbar)\leq \mu_1^\textrm{ sw}(\hbar)+\tilde{\mathcal  O}(e^{-\mathsf{S}^a/\sqrt{\hbar}}),
\end{equation}
where $\mathsf S^a$ is introduced in \eqref{defSa},
\[\mu_1^\textrm{ sw}(\hbar)=\inf\sigma(\mathcal N_{\hbar,r,\gamma_0})= \inf\sigma(\mathcal N_{\hbar,r,\gamma_0})\,,\]
$(\nu_{j,a}(\hbar))_{j\geq 1}$ is the sequence of eigenvalues of the operator $\mathcal N_{\hbar}$, and 
the notation $\tilde{\mathcal  O}(e^{-\mathsf{S}^a/\sqrt{\hbar}})$ means
\begin{equation}\label{eq:not-tO}
 {\mathcal  O}(e^{(\epsilon -\mathsf{S}^a)/\sqrt{\hbar}})\quad\textrm{ for~any~fixed~}\varepsilon>0. \end{equation}
The analysis of the tunneling requires an explicit approximation of the ground state of  the single well 
operators. Recall  the Agmon distance  $\Phi_r$ and the WKB solution $\Psi_{\hbar,r}$ introduced in  \eqref{eq:Ag-r} and  Theorem~\ref{thm:WKB} respectively. Consider the flux free `right well' operator $\mathcal  N_{\hbar,r}:=\mathcal N_{\hbar,r,0}$. By \eqref{eq:r-well}, the low lying eigenvalues of this operator are simple; we denote by $\Pi_r$ the orthogonal projection on its first eigenspace. By \cite[Prop.~6.3]{BHR21}, it results from Theorem~\ref{thm:dec}:

\begin{proposition}[WKB  approximation]\label{prop:WKB}
We have
\[\|\psi_{\hbar,r}-\Pi_r\psi_{\hbar,r}\|_{L^2(\mathbb  R^2)}=\mathcal{O}(\hbar^\infty)\]
and
	\begin{equation}\label{eq:app.WKB1'}
	\langle\tau\rangle\, e^{\Phi_{r}/\sqrt{\hbar}}(\Psi_{\hbar,r}-u_{\hbar,r})=\mathcal
	{O}(\hbar^\infty)\quad\textrm{  in}~\mathscr{C}^1(K;L^2(\mathbb R)),
	\end{equation}
	where $K\subset I_{2\hat\eta,r}:=(s_\ell-2L+\hat\eta,s_\ell-\hat\eta)$ is a compact set, 
	\[\psi_{\hbar,r}(\sigma,\tau):=\chi_{\hat\eta,r}(\sigma)\Psi_{\hbar,\tau}(\sigma,\tau)\,, \]
	and $\chi_{\hat\eta,r}$ is a cut-off function supported in $I_{\hat\eta,r}$  such  that $\chi_{\hat\eta,r}=1$ on  $I_{2\hat\eta,r}$.
\end{proposition}

\section{Interaction matrix and tunneling}\label{sec:IM}
We return to the operator $\mathcal N_{\hbar}$ introduced in \eqref{eq:Nh}.  In order to estimate the splitting between the first and second eigenvalues, $\nu_2(\hbar)-\nu_1(\hbar)$, we will write the matrix of this operator  in a specific basis of 
\[E=\oplus_{i=1}^2\textrm{ Ker}\big(\mathcal N_\hbar-\nu_i(\hbar)\big). \]
Let $\Pi$ be the orthogonal  projection  on $E$.  We introduce the two functions
\[ f_{\hbar,r}=\chi_{\hat\eta,r}\phi_{\hbar,r},\qquad \qquad  f_{\hbar,\ell}=\chi_{\hat\eta,\ell}\phi_{\hbar,\ell},\]
where $\chi_{\hat\eta,r}$ is the cut-off function introduced in Proposition~\ref{prop:WKB},  $\chi_{\hat\eta,\ell}=U\chi_{\hat\eta,r}$ is defined by the symmetry operator (see Sec.~\ref{sec:LWop}),
$\phi_{\hbar,r}$, $\phi_{\hbar,\ell}$ enjoy periodicity  properties and are defined by   inspiration  from the functions in \eqref{eq:gs-r} and \eqref{eq.phil0}.  In fact,  $\phi_{\hbar,r}(\sigma,\tau)$ need to be defined in the support  of $\chi_{\hat\eta,r}$. Starting on $[-L,s_\ell-\frac{\hat\eta}2)\times\R$,   we take $\phi_{\hbar,r}(\sigma,\tau)$    the same as the function in \eqref{eq:gs-r}; on $[s_\ell+\frac{\hat\eta}{2},L)\times\R$,  we do  a change of variable,  and modify the function in \eqref{eq:gs-r} so  that its module satisfies a periodic boundary condition on $\pm L$. More precisely,  we have
\begin{equation}\label{eq:phi-h-r}
\phi_{\hbar, r}(\sigma,\tau) = \begin{cases}
e^{ -i\gamma_0 \sigma/\hbar^2}u_{\hbar,r}(\sigma,\tau),&\mbox{if }-L\leq \sigma\leq s_\ell-\frac{\hat\eta}{2},
\\
e^{ -i\gamma_0 (\sigma-2L)/\hbar^2}u_{\hbar,r}(\sigma-2L,\tau),&\mbox{if } s_\ell + \frac{\hat\eta}{2}<\sigma< L,
\end{cases}
  \end{equation}
  and so $f_{\hbar,r}$ is well defined on $[-L,L)$.  In a similar fashion, 
 \begin{equation}\label{eq:phi-h-l}
\phi_{\hbar, \ell}(\sigma,\tau) = 
\begin{cases}
e^{-i\gamma_0 (\sigma+2L)/\hbar^2}u_{\hbar,\ell}(\sigma+2L,\tau),&\mbox{if }-L\leq \sigma\leq s_r - \frac{\hat\eta}{2},
\\
e^{-i\gamma_0 \sigma/\hbar^2}u_{\hbar,\ell}(\sigma,\tau),&\mbox{if } s_r+\frac{\hat\eta}2<\sigma< L,
\end{cases}
\end{equation}
and $f_{\hbar, \ell}$ is well defined on $[-L,L)$. 

Also we introduce  the following actual bound  states by projecting on the eigenspace $E$,
\[ g_{\hbar,r}=\Pi f_{\hbar,r},\qquad\qquad g_{\hbar,\ell}=\Pi f_{\hbar,\ell}\,.\]
We use the notation $\tilde{\mathcal O}$ in \eqref{eq:not-tO}. By Proposition~\ref{prop:dec-bs}, we have (see \cite[Sec.~7.1]{BHR21} and \cite[Sec.~3]{BHR} for details)
\begin{align*}
\|f_{\hbar,r}\|^2&=1+\tilde{\mathcal  O}(e^{-2\mathsf{S}^a/\sqrt{\hbar}}),\quad\|f_{\hbar,\ell}\|^2=1+\tilde{\mathcal  O}(e^{-2\mathsf{S}^a/\sqrt{\hbar}}), \\
\langle f_{\hbar,r}, f_{\hbar,\ell}\rangle&=\tilde{\mathcal  O}(e^{-\mathsf{S}^a/\sqrt{\hbar}}),
\end{align*}
and
\[\|g_{\hbar,\alpha}-f_{\hbar,\alpha}\|+ \|\partial_\tau(g_{\hbar,\alpha}-f_{\hbar,\alpha})\|=\tilde{\mathcal  O}(e^{-\mathsf{S}^a/\sqrt{\hbar}})\qquad\alpha\in\{r,\ell\}. \]
Now,  we construct an  orthonormal basis $\mathcal B_{\hbar}:=\{\tilde g_{\hbar,r},\tilde g_{\hbar,\ell}\}$ of $E$ from
$\{g_{\hbar,r},g_{\hbar,\ell}\}$ by the Gram-Schmidt process. Let $\mathsf M$ be the { matrix} of $\mathcal N_{\hbar}$ relative to  the basis $\mathcal B_{\hbar}$. Then
\begin{equation}\label{eq:gap}
\nu_2(\hbar)-\nu_1(\hbar)=2|w_{\ell,r}| +\tilde{\mathcal O}(e^{-2\mathsf{S}/\sqrt{\hbar}})\,,\quad w_{\ell,r}=\langle r_{\hbar,\ell},f_{\hbar,r}\rangle,
\end{equation}
where
\[  r_{\hbar,\ell}=(\mathcal N_{\hbar,\ell}-\mu^\textrm{ sw}(\hbar))f_{\hbar,\ell}.\]
All we have to do now is the computation of $w_{\ell,r}$ by the WKB approximation in Proposition~\ref{prop:WKB}. By \cite[Lem.~7.1]{BHR21}(which is essentially an integration by parts formula)
\begin{multline*}w_{\ell,r}=i\hbar\int_\R a_\hbar^{-1}\left(\phi_{\hbar,\ell}\overline{\mathscr{D}_\hbar\phi_r}+[\mathscr{D}_\hbar\phi_{\hbar,\ell}]\overline{\phi_{\hbar,r}}\right)(0,\tau){ d\tau}\\
-i\hbar\int_\R   a_\hbar^{-1}\left(\phi_{\hbar,\ell}\overline{\mathscr{D}_\hbar\phi_{\hbar,r}}+[\mathscr{D}_\hbar\phi_\ell]\overline{\phi_{\hbar,r}}\right)(-L,\tau)d\tau,
\end{multline*}
where
\[ \mathscr{D}_\hbar=\hbar D_\sigma+\hbar^{-1}\gamma_0-b_a(\tau)\tau+\hbar c_\mu \frac{k}{2}b_a(\tau)\tau^2\,.\]
Writing $a_\hbar=1+o(1)$, $\hbar c_\mu\tau^2=o(\hbar^{-2\eta})$, and approximating $\phi_{\hbar,r}$, $\phi_{\hbar,\ell}$ by using \eqref{eq:phi-h-r}, \eqref{eq:phi-h-l} and Proposition~\ref{prop:WKB}, we get (see \cite[Sec.~7.2.2]{BHR21} for details)
\begin{equation}\label{eq:w-r-l}
w_{r,\ell}=i\hbar(w_{r,\ell}^\textsf{ u}+ w_{r,\ell}^\textsf{ d})
\end{equation}
where
\begin{align*}
w_{\ell,r}^\textsf{ u}&=
\int_\R a_\hbar^{-1}\Psi_{\hbar,\ell}\overline{\left(\hbar D_\sigma-b_a\tau+\hbar c_\mu\frac{\kappa}{2}b_a\tau^2\right)\Psi_{\hbar,r}}(0,\tau)d\tau\\
&+\int_\R a_\hbar^{-1}
\left[\left(\hbar D_\sigma-b_a\tau+\hbar c_\mu\frac{\kappa}{2}b_a\tau^2\right)\Psi_{\hbar,\ell}\right]\overline{\Psi_{\hbar,r}}(0,\tau)d\tau
+\mathcal{O}(\hbar^\infty)e^{-\mathsf{S}_{\mathsf{u}}^a/\hbar^{1/2}} ,
\intertext{and}
w_{\ell,r}^\textsf{ d}&=
\int_\R a_\hbar^{-1}\Psi_{\hbar,\ell}\overline{\left(\hbar D_\sigma-b_a\tau+\hbar c_\mu\frac{\kappa}{2}b_a\tau^2\right)\Psi_{\hbar,r}}(-L,\tau)d\tau\\
&+\int_\R a_\hbar^{-1}
\left[\left(\hbar D_\sigma-b_a\tau+\hbar c_\mu\frac{\kappa}{2}b_a\tau^2\right)\Psi_{\hbar,\ell}\right]\overline{\Psi_{\hbar,r}}(-L,\tau)d\tau
+\mathcal{O}(\hbar^\infty)e^{-\mathsf{S}_{\mathsf{d}}^a/\hbar^{1/2}} .
\end{align*}
Eventually, using \eqref{eq:eff-op} and \eqref{eq.psir}, we get (see \cite[Eq.~(7.11)]{BHR21}) 
\[\hbar^{\frac14}e^{\mathsf{S}_{\mathsf{u}}^a/\hbar^{1/2}}w_{\ell,r}^\textsf{ u}=-i \hbar^{\frac12}\mu_a''(\zeta_a) \pi^{-\frac 12} g_a^{\frac12}\sqrt{V_a(0)} \mathsf{A}_{\mathsf{u}}^ae^{-2i\alpha_0(0)}+\mathcal {O}(\hbar) \]
and, in a similar  fashion,
\begin{multline*}
\hbar^{\frac14}e^{\mathsf{S}_{\mathsf{d}}^a/\hbar^{1/2}}w_{\ell,r}^\textsf{ d}\\
=-i \hbar^{\frac12}\mu_a''(\zeta_a) \pi^{-\frac 12} g_a^{\frac12}\sqrt{V_a(-L)} \textsf{A}_{\mathsf{d}}^ae^{-2i\alpha_0(-L)}e^{i ( -2L\gamma_0/\hbar^2-2L\zeta_a/\hbar)}+\mathcal {O}(\hbar) \,,\end{multline*}
 where $\alpha_0$  is the function introduced in Remark~\ref{eq.effectiveT}.
 
 Collecting \eqref{eq:w-r-l} and \eqref{eq:gap} and using that $\hbar=h^{1/2}$, we get
 \[\begin{aligned}
  \nu_2(\hbar)-\nu_1(\hbar)&=2|e^{iLf(h)}w_{\ell,r}|+\tilde{\mathcal O}(e^{-2\mathsf S^a/\sqrt{\hbar}})\\
  &=h^{-1}| w_a(h)|+\tilde{\mathcal O}(h^{-1}e^{-2\mathsf S^a/\sqrt{\hbar}}),
  \end{aligned}\]
  where $f(h)$  and  $\tilde w_a(h)$ are the expressions in Theorem~\ref{thm:FHK}. In light of 
  Proposition~\ref{prop:BHR2.7},  this  finishes the proof of Theorem~\ref{thm:FHK}.
\section{Conclusion and open problems}

 Until  now, examples of magnetic tunneling effects  are rare in the literature.  Very few articles have been dealing with the measure of the tunneling effect due to the presence of the magnetic field. In  the presence of an electric potential with multiple wells, the article \cite{HeSj7}  was only considering   a case when the
 magnetic field was a perturbation and the tunneling  was  mainly  created by the electric potential. Other examples  include the case of a pure flux \cite{KR17}. After the recent contributions of Bonnaillie-H\'erau-Raymond  \cite{BHR}  and  Fefferman-Shapiro-Weinstein \cite{FSW} (see also references therein), we have presented  a new magnetic tunneling effect
  due to the curvature of the magnetic edge.
  
 Both for the Neumann problem occurring in surface superconductivity \cite{HM3D} or for the problem considered here \cite{AG}, it would be interesting to consider the $(3D)$-case. 
  
  Excluded in this paper is the case $a=-1$, where localization near  the point(s) of maximum curvature no more occurs ($M_3(a)=0$ in the asymptotics of Theorem~\ref{thm:AHK}).  In   contrast, this  case seems to feature an interesting new phenomenon where localization near the whole edge $\Gamma$ occurs, which  also has a nice analogy  to what was observed in the multiple wells situation in \cite{HeSj6}. We hope to come back to the treatment of this case rather soon.
  
  Finally we mention that the standard purely magnetic double well problem seems at the moment a difficult  challenge. Here we consider a purely semi-classical magnetic Laplacian (say in $\mathbb R^2$)   where the magnetic 
   field has two symmetric non degenerate positive minima.

\section{Appendix: Regularized resolvent and moments}

Let us return back to the flat  edge model in Sec.~\ref{sec:FlatEdge}  and recall some necessary computational results.

We  can invert the operator $\mathfrak h_a[\zeta_a]$ on the orthogonal complement of the ground state $\phi_a$. Extending by  linearity, we get  the regularized resolvent $\mathfrak R_a$ defined on  $L^2(\mathbb R )$ by
	\begin{equation}\label{eq:R}
	\mathfrak R_a\, u =
	\begin{cases}
	0&\mathrm{if}~u\parallel\phi_a\\
	(\mathfrak h_a[\zeta_a]-\beta_a)^{-1}u&\mathrm{if}~u\perp \phi_a
	\end{cases}\,.
	\end{equation}
By \cite{AK20}, $(\zeta_a+b_a(\tau) \tau)\phi_a$ and  $\phi_a$ are orthogonal in $L^2(\mathbb R)$:
\begin{equation}\label{eq:m0}
\int_{\mathbb R}(\zeta_a+b_a(\tau) \tau)|\phi_a(\tau)|^2\,d\tau=0\,.
\end{equation} 
Later on we will encounter the following integral \cite[Prop.~2.5]{AHK}
\begin{equation}\label{eq:I2}  
I_2(a):=\int_{\mathbb R }\phi_a(\tau) \mathfrak R_a[(\zeta_a+b_a(\tau) \tau)\phi_a]\,d\tau
 =\frac14-\frac{\mu_a''(\zeta_a)}{8}\,.
\end{equation}	
We  recall   some identities from \cite{AK20}  involving for $n\in \mathbb N $  quantities of the form 
\begin{equation}\label{eq:moments}
M_n(a)=\int_{\mathbb R}\frac 1{b_a(\tau)}(\zeta_a+b_a(\tau)\tau)^n|\phi_a(\tau)|^2\,d\tau\,. \end{equation}
We have
	\begin{align}
	M_1(a)&=0\,,\label{eq:m1}\\
	M_2(a)&=-\frac 12 \beta_a\int_{\mathbb R}\frac 1{b_a(\tau)}|\phi_a(\tau)|^2\,d\tau+\frac 14\Big(\frac 1a-1\Big)\zeta_a\phi_a(0)\phi_a'(0)\,, \label{eq:m2}\\
	M_3(a)&=\frac 13\Big(\frac 1a-1\Big)\zeta_a\phi_a(0)\phi_a'(0)\,.\label{eq:m3*}
	\end{align}
The case $a=-1$ is special  because
\[M_3(-1)=0 \quad \mathrm{and}\quad M_3(a)<0 \quad \mathrm{for} \quad -1<a<0\,.\]
\begin{remark}\label{rem:AHK2.3}
The next identities   follow in a straightforward manner \cite[Rem.~2.3]{AHK},
\begin{align*}
\int_{\mathbb R}\tau(\zeta_a+b_a(\tau) \tau)|\phi_a(\tau)|^2\,d\tau&=M_2(a),\\
\int_{\mathbb R}\tau(\zeta_a+b_a(\tau) \tau)^2|\phi_a(\tau)|^2\,d\tau&=M_3(a) -\zeta_aM_2(a),\\
\int_{\mathbb R}b_a(\tau) \tau^2(\zeta_a+b_a(\tau) \tau)|\phi_a(\tau)|^2\,d\tau&=M_3(a)-2\zeta_aM_2(a),
\end{align*}
and
\begin{align*}
\int_{\mathbb R}\tau|\phi_a(\tau)|^2\,d\tau&=-\zeta_a\int_{\mathbb R}\frac 1{b_a(\tau) }|\phi_a(\tau)|^2\,d\tau,\\
\int_{\mathbb R}\tau|\phi'_a(\tau)|^2\,d\tau&=\beta_a\zeta_a\int_{\mathbb R}\frac 1{b_a(\tau) }|\phi_a(\tau)|^2\,d\tau+2M_3(a) -2\zeta_aM_2(a).
\end{align*}

\end{remark}
\subsection*{Acknowledgments} The authors would like to thank the anonymous referee for  the attentive reading of this paper and  the valuable comments.
\subsection*{Funding}
Part of this work was done while SF and AK visited the Laboratoire Jean Leray at the university of Nantes.  They wish to thank \emph{la f\'ed\'eration de recherche math\'ematique des Pays de Loire} and  \emph{Nantes Universit\'e} for supporting the  visit.

\end{document}